\documentclass[a4paper, 12pt]{amsart}

\usepackage{amsmath,amsthm,amssymb,amsfonts,calligra}
\usepackage{tikz, pgfplots}
\usepackage{tikz-cd} 
\usepackage[a4paper]{geometry}
\usepackage{bm} 
\usepackage{graphicx}
\usepackage{mathrsfs} 
\usepackage{enumitem} 
\usepackage{stmaryrd}
\usepackage{footnote} 
\makesavenoteenv{tabular}
\usepackage{color}
\usepackage{xcolor}
\usepackage[colorlinks=true]{hyperref} 
\usepackage[final]{showlabels} 
\usepackage{caption}

\mathchardef\mhyphen="2D

\theoremstyle{plain}
\newtheorem{thm}{Theorem}
\newtheorem{prop}{Proposition}
\newtheorem{lem}{Lemma}
\newtheorem{cor}{Corollary}

\theoremstyle{definition}

\newtheorem{ex}{Example}
\newtheorem*{question*}{Question}

\theoremstyle{remark}
\newtheorem{rem}[thm]{Remark}
\newtheorem*{claim}{Claim}

\DeclareMathOperator{\Aut}{Aut}
\DeclareMathAlphabet{\mathcalligra}{T1}{calligra}{m}{n}
\DeclareMathOperator{\Hom}{Hom}

\newcommand{\Q}{\mathbb{Q}}
\newcommand{\C}{\mathbb{C}}

\newcommand{\Z}{\mathbb{Z}}

\newcommand{\Oo}{\mathcal{O}}

\newcommand{\X}{\mathcal{X}} 
\newcommand{\Xo}{\mathcal{X}^\circ}
\newcommand{\cu}{\Sigma} 
\newcommand{\Bb}{\mathbf{B}} 
\newcommand{\Bbo}{\mathbf{B}^\circ} 

\newcommand{\pib}{\bm{\pi}} 

\newcommand{\Hig}{\mathbf{Higgs}}
\newcommand{\btcu}{\bm{\tilde{\Sigma}}} 
\newcommand{\tcu}{\tilde{\Sigma}} 
\newcommand{\bsigma}{\bm{\sigma}} 
\newcommand{\gfr}{\mathfrak{g}} 

\newcommand{\tfr}{\mathfrak{t}} 

\newcommand{\Ub}{\bm{U}} 
\newcommand{\Ubt}{\bm{\tilde{U}}} 

\newcommand{\Hit}{\bm{h}}

\newcommand{\ev}{ev} 

\newcommand{\Cc}{\mathbf{C}} 
\newcommand{\CcA}{\mathbf{CA}} 
\newcommand{\ADE}{\mathrm{ADE}}
\newcommand{\BCFG}{\mathrm{BCFG}}

\newcommand{\btSigma}{\bm{\tilde{\Sigma}}}

\newcommand{\btsigma}{\tilde{\bm{\sigma}}}
\newcommand{\bomega}{\bm{\omega}}

\newcommand{\tSc}{\tilde{\mathcal{S}}}
\newcommand{\Sc}{\mathcal{S}}
\newcommand{\Scb}{\bm{S}}
\newcommand{\tScb}{\tilde{\bm{S}}}

\newcommand{\HomS}{\mathscr{H}\hspace{-5pt}\mathcalligra{om}} 
\newcommand{\ExtSO}{\mathscr{E}\hspace{-3.3pt}\mathcalligra{xt}_{\,\mathcal{O}_X}} 

\begin{document}
\begin{abstract}

Any irreducible Dynkin diagram $\Delta$ is obtained from an irreducible Dynkin diagram $\Delta_h$ of type $\ADE$ by folding via graph automorphisms. 
For any simple complex Lie group $G$ with Dynkin diagram $\Delta$ and compact Riemann surface $\cu$, we give a Lie-theoretic construction of families of quasi-projective Calabi-Yau threefolds together with an action of graph automorphisms over the Hitchin base associated to the pair $(\cu, G)$ . 
These give rise to Calabi-Yau orbifolds over the same base. 
Their intermediate Jacobian fibration, constructed in terms of equivariant cohomology, is isomorphic to the Hitchin system of the same type away from singular fibers. 
\end{abstract}

\title{Calabi-Yau orbifolds over Hitchin bases}
\author{Florian Beck}
\address{FB Mathematik, Universit\"at Hamburg, Bundesstrasse 55, 20146 Hamburg, Germany}
\email{florian.beck@uni-hamburg.de}
\maketitle

\tableofcontents

\section{Introduction}
In recent years, it has been realized that the $G$-Hitchin system (\cite{Hit1}, \cite{Hit2})
\begin{equation*}
\Hit:\mathcal{M}_{\mathrm{H}}(\cu, G)\to \Bb(\cu, G), 
\end{equation*}
for a compact Riemann surface $\cu$ of genus $\geq 2$ and a simple complex Lie group $G$ of type $\ADE$, 
are closely related to certain
(quasi-projective) Calabi-Yau threefolds (\cite{DDD}, \cite{DDP}, \cite{Tbranes}, \cite{Tbranes2}).   
For example, if $G$ is a simple adjoint complex Lie group of type $\ADE$, then there exists a family $\pib:\X\to \Bb(\cu,G)$ of quasi-projective Gorenstein Calabi-Yau threefolds over the Hitchin base $\Bb(\cu, G)$ such that the intermediate Jacobian 
\begin{equation*}
J^2(X_b)=H^3(X_b,\C)/\left(F^2H^3(X_b,\C)+H^3(X_b,\Z)\right),\quad X_b:=\pib^{-1}(b), 
\end{equation*}
is isomorphic to the Hitchin fiber $\Hit^{-1}(b)$ for generic $b\in \Bb(\cu,G)$ (\cite{DDP}). 
In this paper, we extend this to a global result for all simple adjoint and simply-connected complex Lie groups of any Dynkin type by constructing families of Calabi-Yau orbifolds/global orbifold stacks\footnote{Since the Hitchin base $\Bb(\cu, G)$ only depends on the Dynkin diagram $\Delta(G)$, it is sufficient to write $\Bb(\cu, \Delta)$, see \ref{ss:hitchinbase}).}
\begin{equation}\label{intro:fam}
[\X/\Cc]\to \Bb(\cu, \Delta)=\Bb(\cu,G), \quad \Delta=\Delta(G).
\end{equation}
Here $\Delta(G)$ is the irreducible Dynkin diagram associated to $G$ which is obtained by a unique irreducible Dynkin diagram $\Delta_h$ of type $\ADE$ by folding via graph automorphisms $\Cc\subset \Aut(\Delta_h)$, i.e. $\Delta=\Delta_h^{\Cc}$.
Moreover, $\pib:\X\to \Bb(\cu, \Delta)$ is a family of quasi-projective Gorenstein Calabi-Yau threefolds with $\Cc$-trivial canonical class. 
More precisely, we prove in Section \ref{s:mhs}:
\begin{thm}\label{thm:thm1}
Let $[\X/\Cc]\to \Bb=\Bb(\cu,\Delta)$ be a family of global Calabi-Yau orbifolds as in (\ref{intro:fam}).
Its intermediate Jacobian fibration 
\begin{equation*}
\mathcal{J}^2([\X^\circ/\Cc]) \to \Bb^\circ, 
\end{equation*}
defined in integral equivariant cohomology, is isomorphic to the $G_{ad}(\Delta)$-Hitchin system over a Zarisiki-open and dense $\Bb^\circ \subset \Bb$ where $G_{ad}(\Delta)$ is the simple adjoint complex Lie group with Dynkin diagram $\Delta$. 
\end{thm}
The basic idea for constructing $\X\to \Bb$ (without the $\Cc$-action) goes back to \cite{Sz1} and \cite{DDP}. 
However, we offer two main novelties. 
The first one is a consequent use of Lie-theoretic methods, namely Slodowy slices, which makes the link to Hitchin systems from the start. 
For example, we construct nowhere-vanishing sections $s_b\in H^0(X_b,K_{X_b})$ of the canonical class $K_{X_b}$ of $X_b$, $b\in \Bb$, in terms of the Kostant-Kirillov form. 
The second one is the construction of a $\Cc$-action on the family $\X \to \Bb$ and we prove that the sections $s_b$ are $\Cc$-invariant. 
Therefore the global quotient family $[\X/\Cc]\to \Bb$ is indeed a family of Calabi-Yau orbifolds. 
\\
Another important step in the proof of Theorem \ref{thm:thm1} is an isomorphism 
\begin{equation}\label{eq:isoMHS}
H^3_{\Cc}(X_b,\Z) \cong H^3(X_b,\Z)^{\Cc}, \quad b\in \Bb^\circ,
\end{equation}
between Deligne's $\Z$-MHS on the integral equivariant cohomology $H^3_{\Cc}(X_b,\Z)$ and the $\Z$-MHS on the $\Cc$-invariant part $H^3(X_b,\Z)^{\Cc}$ of the integral cohomology. 
Even though (\ref{eq:isoMHS}) always holds true over $\Q$, it is false in general over $\Z$ due to torsion. 
For example, it fails for the minimal resolution of the $\Delta_h$-singularity (Example \ref{ex:eqcoh}). 
Combined with the global isomorphism
\begin{gather}\label{eq:J2C}
\mathcal{J}^2_{\Cc}(\X^\circ) \cong \mathcal{M}_{\mathrm{H}}(\cu, G_{ad}(\Delta)), 
\end{gather}
over $\Bbo$ (Corollary 5 in \cite{Beck}), where $\mathcal{J}_{\Cc}^2(\X^\circ)\to \Bbo$ is the intermediate Jacobian fibration defined by $\Cc$-invariants in cohomology, we conclude Theorem \ref{thm:thm1}. 
In fact, the isomorphism (\ref{eq:J2C}) gave a hint that we have to use with Calabi-Yau orbifolds if $\Cc\neq 1$ instead of ordinary Calabi-Yau threefolds if $\Cc\neq 1$. 
By working with equivariant compactly supported cohomology instead of equivariant cohomology, we further obtain the Langlands dual version of Theorem \ref{thm:thm1} and hence cover all simple simply-connected complex Lie groups as well. 
\\
Besides the $\Cc$-action and construction of $\Cc$-invariant holomorphic volume forms, the Lie-theoretic construction gives a way to study the smooth locus of the family $\X\to \Bb$. 
As an application, we see in Corollary \ref{cor:comp} that if $\Delta_h$ is of type $\ADE$ and admits non-trivial Dynkin graph automorphisms $\Cc$ (see Remark \ref{r:dynkin} in Appendix \ref{SectFolding}), then the smooth locus of the family $\X\to \Bb(\cu, \Delta_h)$ is much larger than the smooth locus of $\Delta_h$-cameral curves. 
It follows that the family $\X\to \Bb(\cu, \Delta_h)$ does not only geometrically encode the smooth loci of 
\begin{equation*}
\begin{tikzcd}
\mathcal{M}_{\mathrm{H}}(\cu, G_{ad}(\Delta_h)) \ar[r] & \Bb(\cu, \Delta) & \ar[l] \mathcal{M}_{\mathrm{H}}(\cu, G_{sc}(\Delta_h))
\end{tikzcd}
\end{equation*}
but correspondingly for $G_{ad}(\Delta)$ and $G_{sc}(\Delta)$, under the natural inclusion $\Bb(\cu, \Delta)\subset \Bb(\cu, \Delta_h)$ for $\Delta=\Delta_h^\Cc$.

\subsection{Relation to other works}
Our work is intellectually indebted to \cite{Sz1}, \cite{DDP}, \cite{DDD} and \cite{Slo}.
We relate our approach to the first three works, for example we give explicit equations in the spirit of \cite{DDD} (also see Theorem \ref{thm:defX0}). 
For that reason, we give two Slodowy slices $S\subset \mathfrak{so}(5,\C) $ and $S_h \subset \mathfrak{sl}(4,\C) $  and show by a direct computation that they both realize the semi-universal deformation of the $\mathrm{B}_2$-singularity in Appendix \ref{a:ex}.
Moreover, we show that families of \cite{Sz1} provide simultaneous resolutions over a cone in the singular locus of $\pib:\X\to \Bb$ (Proposition \ref{p:simresol}). 
Graph automorphisms also appear in \cite{Sz1} but we use them differently. 
In particular, orbifold stacks and their cohomology appear in neither of these works. 
Finally, it is further crucial for our work \cite{Beck} that the `volume forms' $s_b\in H^0(X_b,K_{X_b})$ induce a period map on $\Bbo$.
This period map is a so-called abstract Seiberg-Witten differential which in turn determines the structure of an algebraic integrable system on $\mathcal{J}_\Cc^2(\X^\circ/\Bbo)\to \Bbo$ (cf. Section 4.4 in \cite{Beck}).





\subsection{Outline}
Since the key ingredient for the local building blocks of the Calabi-Yau orbifolds (\ref{intro:fam}) are Slodowy slices, we recall and prove important preparatory results on them in Section \ref{s:local}.
Section \ref{s:cy} begins with a recap on Hitchin bases and cameral curves and continues with the construction of Calabi-Yau orbifolds and related families of quasi-projective Calabi-Yau threefolds. 
In Section \ref{s:mhs} we study mixed Hodge structures on equivariant cohomology and prove Theorem \ref{thm:thm1}. 
Appendix \ref{SectFolding} contains our conventions on folding of Dynkin diagrams and Appendix \ref{a:ex} gives two examples of Slodowy slices that realize the semi-universal deformations of the same surface singularity. 

\subsection{Acknowledgements}
It is a pleasure to thank Katrin Wendland and Emanuel Scheidegger as well as Ron Donagi and Tony Pantev for discussions and questions on this paper's content. 
This work was supported by the DFG Graduiertenkolleg 1821 "Cohomological Methods in Geometry" 
and through the DFG Emmy-Noether grant on "Building blocks of physical theories from the geometry of quantization and BPS states", number AL 1407/2-1.

\section{Local preparations}\label{s:local}
\subsection{$\Delta$-singularities and Slodowy slices}
Let $\Delta$ be an irreducible Dynkin diagram.
Then there is a unique pair\footnote{The subindex of $\Delta_h$ stands for homogeneous.} $(\Delta_h,AS(\Delta))$ consisting of an $\ADE$-Dynkin diagram and the associated symmetry group $AS(\Delta)\subset \Aut(\Delta_h)$ such that $\Delta$ is obtained by folding, i.e. $\Delta=\Delta_h^{AS}$ (see Appendix \ref{SectFolding} for more details). 
A $\Delta$-singularity $(Y,H)$ (\cite{Slo}) consists of a (germ of a) surface singularity $Y=(Y,0)$ of type $\Delta_h$ and a subgroup $H\subset \Aut(Y)$ with the following properties: 
\begin{enumerate}[label=\roman*)]
\item 
$H\cong AS(\Delta)$; 
\item
the action of $H$ on $Y-\{0\}$ is free; 
\item 
the induced action on the dual resolution graph of the minimal resolution $\hat{Y}\to Y$ coincides with the $AS(\Delta)$-action on $\Delta_h$. 
\end{enumerate}
For $\Delta=\Delta_h$, these are just $\ADE$-surface singularities. 
Every $\Delta$-singularity $(Y,H)$ is quasi-homogeneous (cf. Table \ref{ADEsings} in Appendix \ref{SectFolding}), in particular, $Y$ carries a $\C^*$-action that commutes with the $H$-action. 
A $\C^*$-deformation of $(Y,H)$ is therefore a $\C^*\times H$-deformation $\mathcal{Y}\to B$ such that $H$ acts trivially on the base. 
Each $(Y,H)$ has a semi-universal $\C^*$-deformation, see Section $2$ in \cite{Slo}.
\\
The relation to the simple complex Lie algebra $\gfr=\gfr(\Delta)$ of type $\Delta$ is as follows: 
Let $x\in \gfr$ be a subregular nilpotent element and $(x,y,h)$ an $\mathfrak{sl}_2$-triple in $\gfr$ with semisimple $h$. 
Then the Slodowy slice through $x$ (associated to the triple $(x,y,h)$) is given by 
\begin{equation*}
S=S(\Delta):=x+\ker \mathrm{ad}(y).
\end{equation*}
It carries a non-trivial action by the group $\C^*\times AS(\Delta)$ (cf. Section 6 in \cite{Slo}). 
Here the $AS(\Delta)$-action is defined by the action (via adjunction) of the group\footnote{We often use $\Cc$ to emphasize how the $AS(\Delta)$-action is realized.}
\begin{equation}\label{eq:Cc}
\Cc= C(x,h)/C(x,h)^\circ \cong AS(\Delta)
\end{equation}
which is the group of connected components of $C(x,h)=\{g\in G_{ad}~|~g\cdot x=x, g\cdot h= h\}$. 
For a fixed Cartan subalgebra $\tfr\subset \gfr$, denote by $\chi:\gfr\to \tfr/W$ the adjoint quotient and 
\begin{equation*}
\sigma:=\chi_{|S}: S\to \tfr/W
\end{equation*}
its restriction to the Slodowy slice $S$. 
If we let $\C^*$ act on $\tfr/W$ with \emph{twice} the usual weight and $\Cc$ on $\tfr/W$ trivially, then $\sigma$ is $\C^*\times \Cc$-equivariant. 
Slodowy has shown that $\sigma: S \to \tfr/W$ with the $\C^*\times \Cc$-action is a $\C^*$-semi-universal deformation of the $\Delta$-singularity $(\sigma^{-1}(\bar{0}),x)$ (\cite{Slo}, Section 8.7). 
\\
Finally, we recall Grothendieck's simultaneous resolution of $\chi:\gfr\to \tfr/W$,
\begin{equation}\label{eq:Gsimresol}
\begin{tikzcd}
\tilde{\gfr}\cong G\times^B \mathfrak{b} \ar[r, "\psi"]  \ar[d, "\theta"'] & \gfr \ar[d, "\chi"] \\
\tfr \ar[r , "q"] & \tfr/W. 
\end{tikzcd}
\end{equation}
If $S=x+\ker \mathrm{ad}(y)$ is a Slodowy slice, then (\ref{eq:Gsimresol}) restricts to the simultaneous resolution $\psi:\tilde{S}\to S$ of $\sigma: S\to \tfr/W$ (see Chapter 5 in \cite{Slo}).
If $(x,y,h)$ is the $\mathfrak{sl}_2$-triple defining $S$, then the resulting square is a square of $\C^*$-spaces if $h\in \tfr$ (see Remark 1.53 in \cite{Beck-thesis} for details). 
Since any two Cartan subalgebras are conjugate to each other, we will assume $h\in \tfr$ from now on. 


\subsection{Relative symplectic form}

It is well-known (e.g. Chapter 1, \cite{ChrissGinzburg}) that the the Kostant-Kirillov form $\nu=\omega_{KK}$  restricts to a symplectic form on each adjoint orbit $O\subset \gfr$. 
Consequently, the relative differential form $\hat{\nu}\in \Omega^2_{\chi}(\gfr)$ induced by $\nu$ restricts to a relative symplectic form on the locus $\gfr^{reg}$ of regular elements in $\gfr$.
This holds true for the restriction $\hat{\nu}_{|S^{reg}}\in \Omega^2_{\sigma^{reg}}$ to $S^{reg}:=S\cap \gfr^{reg}$ as well (cf. \cite{GG}, Appendix 7).
Since $S-S^{reg}$ is of codimension at least $2$, $\hat{\nu}_{|S^{reg}}$ extends to a global section of $K_{\sigma}$ which we also denote by $\hat{\nu}\in \Gamma(S,K_{\sigma})$.  
\begin{prop}\label{p:C*equivariant}
Let $S \subset \gfr$ be a Slodowy slice and $\hat{\nu}\in \Gamma(S,K_{\sigma})$ be the previously constructed section. 
Then $\hat{\nu}$ is nowhere vanishing, in particular $K_\sigma$ is trivializable. 
Moreover, it is $\C^*$-equivariant, 
\begin{equation*}
\lambda^*\hat{\nu}=\lambda^2 \hat{\nu},\quad \lambda\in \C^*, 
\end{equation*}
as well as $\Cc$-invariant.
\end{prop}

\begin{proof}
The proof proceeds in two steps. 
In the first step, we use Yamada's realization (\cite{Yamada}) of $\theta: \tilde{\gfr}\to \tfr$ as a relative symplectic reduction to relate the corresponding relative symplectic form $\hat{\omega} \in \Omega^2_{\theta}(\tilde{\gfr})$ to $\hat{\nu}\in \Omega^2_\chi(\gfr)$ over the locus $\gfr^{reg}\subset \gfr$ of regular elements. 
In the second step, we restrict to $S^{reg}$ and extend over the non-regular locus. 
\\
For the first step, fix a maximal torus $T$ and a Borel subgroup $B\supset T$ in the adjoint group $G=G_{ad}$ and denote by $N\subset B$ the nilradical of $B$. 
Then the symplectic manifold $(M:=T^*(G/N),\omega_c)$, with the canonical symplectic structure $\omega_c$, carries a natural Hamiltonian $G$- and $T$-action induced by left and right multiplication on $G/N$ respectively. 
The corresponding moment maps descend to $\hat{\mu}_G:M/T \to \gfr^*$ and $\hat{\mu}_T:M/T \to \tfr^*$. 
Yamada (\cite{Yamada}) identifies the square (\ref{eq:Gsimresol}) with the diagram 
\begin{equation*}
\begin{tikzcd}
M/T \ar[r, "\hat{\mu}_G"] \ar[d, "\hat{\mu}_T"']   & \gfr \ar[d, "\chi"] \\
\tfr \ar[r]  & \tfr/W
\end{tikzcd}
\end{equation*}
after $\gfr^*\cong \gfr$ and $\tfr^*\cong \tfr$ via the Killing form (which we will do from now on). 
Under this identification, the symplectic form $\omega_c$ descends and induces the relative symplectic form $\hat{\omega}\in \Omega^2_{\theta}(\tilde{\gfr})$. 
Moreover, for each $t\in \tfr$ the induced $G$-action on $(\theta^{-1}(t),\hat{\omega}_t)$ is Hamiltonian and the restriction $\psi_t:\theta^{-1}(t) \to \gfr$ of $\psi: \tilde{\gfr}\to \gfr$ is the corresponding moment map (Theorem 2.5 in \cite{Yamada}). 
Now let $\psi^{reg}:\tilde{\gfr}^{reg}\to \gfr^{reg}$ be the restriction of $\psi: \tilde{\gfr} \to \gfr$ to the regular locus. 
Note that 
\begin{equation}
\psi^{reg}_t: \tilde{\gfr}^{reg}\cap \theta^{-1}(t) \to \gfr^{reg}\cap \chi^{-1}(\bar{t}),\quad q(t)=\bar{t},
\end{equation}
is an isomorphism for each $t\in \tfr$.
In fact, it is not difficult to see that $\psi^{reg}_t$ is a symplectomorphism if the left-hand side is endowed with the symplectic form $\hat{\omega}_{t}$ and the right-hand side with the Kostant-Kirillov form $\nu$. 
Since $\psi^{reg}$ is smooth, 
$
(\psi^{reg})^*\Omega_\chi^2\cong \Omega^2_\theta
$ 
naturally.
Under this isomorphism, we have
\begin{equation*}
\psi^*\hat{\nu}=\hat{\omega}\in \Gamma(\tilde{\gfr}^{reg}, \Omega^2_\theta)
\end{equation*}
because $\psi^{reg}_t$ is a symplectomorphism for each $t\in \tfr$. 
\\
For the second step, we note that the restriction of $\hat{\omega}$ to $\tilde{S}\subset \tilde{\gfr}$ is again a relative symplectic form (Theorem 4.5 in \cite{Yamada}), again denoted by $\hat{\omega}\in \Omega^2_{\tilde{\sigma}}(\tilde{S})$. 
Again we have a natural isomorphism
\begin{equation*}
\Phi^{reg}:\psi^*K_{\sigma^{reg}}=\psi^*\Omega^2_{\sigma^{reg}}\to \Omega^2_{\tilde{\sigma}^{reg}}
\end{equation*} 
over $\tilde{S}^{reg}$ with $\Phi^{reg}(\psi^*\hat{\nu})=\hat{\omega}$.
We claim that $\Phi^{reg}$ extends to an isomorphism $\Phi: \psi^*K_{\sigma} \to K_{\tilde{\sigma}}$ satisfying 
\begin{equation}\label{eq:nuomega}
\Phi(\psi^*\hat{\nu})=\hat{\omega}.
\end{equation}
Since $K_\sigma$ and $K_{\tilde{\sigma}}$ are reflexive, it is sufficient to prove $\mathrm{codim}_{\tilde{S}}\tilde{T}\geq 2$ for $\tilde{T}:=\tilde{S}-\tilde{S}^{reg}$. 
The components of $\tilde{T}$ of highest dimension lie over the hypersurfaces $\tfr_{\alpha}-\cap_{\beta\neq \alpha} \tfr_\beta\cap \tfr_\alpha$. 
If $t$ lies in such a hypersurface, then $\tilde{\sigma}^{-1}(\bar{t})\cap \tilde{T}$ consists of the exceptional divisor of $\psi_t:\tilde{S}_t\to S_{\bar{t}}$ for $\bar{t}=q(t)$. 
Hence these components of $\tilde{T}$ have dimension $(r-1)+1=r$ which is of codimension $2$ in $\tilde{S}$.
Equation (\ref{eq:nuomega}) implies that $\hat{\nu}$ is nowhere vanishing because $\hat{\omega}$ is and $\psi$ is surjective. 
\\
Finally, the $\C^*$-equivariance follows from that of $\hat{\omega}$ (Corollary 4.6 in \cite{Yamada}) whereas the $\Cc$-invariance is a consequence of the $\Aut(\gfr)$-invariance of the Kostant-Kirillov form $\nu=\omega_{KK}$.
\end{proof}
The morphism $\tilde{\sigma}: \tilde{S}\to \tfr$ is topologically trivial (cf. \cite{Slo2}). 
Hence parallel transport defines canonical isomorphisms
$P_t: H^2(\tilde{S}_t,\C) \to H^2(\tilde{S}_0, \C)$, $ t\in \tfr$. 

\begin{cor}
The period map 
\begin{equation}\label{eq:H2per}
P_{\tilde{S}}:\tfr \to H^2(\tilde{S}_0,\C)^\Cc,\quad t\mapsto P_t([\hat{\omega}_t]),
\end{equation}
is a $W$-equivariant isomorphism.
\end{cor}
\begin{proof}
The case of $\ADE$-Dynkin diagrams, i.e. $\Cc=1$, goes back to \cite{Yamada}. 
For $\Cc\neq 1$, it has been checked in \cite{Beck-thesis}, Corollary 1.98. 
\end{proof}

\subsection{Derivative of the adjoint quotient}\label{SloSectRemOnDerivatives}

Let $b:U\to \tfr/W$ be a morphism, where $U\subset \C$ is a Zariski-open 
subset. 
Then we define $\tilde{U}_b$ and $Y_b$ by the fiber products
\begin{equation}\label{LocalCY3Cameral}
\begin{tikzcd}
\tilde{U}_b \arrow[r] \arrow[d, "p_b"'] & \tfr \arrow[d, "q"] \\
U \arrow[r, "b"] & \tfr/W  \\ 
Y_b \arrow[r] \arrow[u, "\pi_b"] & S. \arrow[u, "\sigma"']
\end{tikzcd}
\end{equation}
In Section \ref{s:cy} it will become clear that $\tilde{U}_b$ is a local model for a cameral curve and $Y_b$ is a local model for the quasi-projective Calabi-Yau threefolds that we construct in loc. cit. 
Both $\tilde{U}_b$ and $Y_b$ are non-singular if $b$ is transversal to $q$ and $\sigma$. 
This condition in particular implies that the rank of $dq$ may not be less than $r-1$ for $r=\mathrm{rk}(\gfr)$. 
\begin{ex}\label{ExSL2}
We consider the simplest example, $\gfr=\mathfrak{sl}_2(\C)$. 
Of course, a Slodowy slice $S$ has to be all of $\gfr$, $\tfr=\C\cong \tfr/W$ and 
\begin{gather*}
q:\C\to \C, \quad z\mapsto z^2,\\
\sigma: S\to \C, \quad A\mapsto \det(A). 
\end{gather*}
A morphism $b:U\to \tfr/W$ is transversal to $q$ iff it has only zeros of multiplicity $1$. 
If this is the case $\tilde{U}_b=\{ (x,y)\in U\times \C~|~y^2-b(x)=0\}$ is non-singular and a branched double covering of $U$. 
It branches precisely over the zeros of $b$. \\
Identifying $\mathfrak{sl}_2(\C)\cong \C^3$, we write $\sigma(u,v,w)=-u^2-vw$. 
Again using that $b$ has simple zeros only, we see that 
$Y_b=\{ ((u,v,w),x)\in \C^3\times U~|~-u^2-vw-b(x)=0\}$
is non-singular. 
Note however, that if $b(x)=0$, then $\pi^{-1}_b(x)$ is an $\mathrm{A}_1$-singularity. 
\end{ex}



\begin{lem}\label{RKdq}
Let $x=h$ be semisimple and $t\in \tfr$ with $\chi(h)=q(t)$. 
Then $\mathrm{im}(dq_t)\subset \mathrm{im}(d\chi_h)$ and
\begin{equation*}
\mathrm{rk}(dq_t)=\dim \bigcap_{\alpha \in R, \alpha(t)=0} \ker \alpha=\dim C(Z_\gfr(h))
\end{equation*}
where $R$ are the roots corresponding to $\tfr\subset \gfr$. 
\end{lem}
\begin{proof}
Let $\tfr(h)\subset \gfr$ be a Cartan subalgebra that contains $h$. 
Since $\tfr(h)$ is conjugate to $\tfr$ we may assume that $h\in \tfr$ and in fact even $t=h$ by the Ad-invariance of $\chi$ and $q$. 
The first claim is obvious because $\chi_{|\tfr}=q$. \\ 
For the second claim we may assume $t=h$ as before. 
Then the first equality is proven in \cite{Steinberg}. 
For the second equality we claim 
\begin{align}\label{CenterOfCentralizer}
\bigcap_{\alpha \in R, \alpha(t)=0} \ker \alpha=C(Z_\gfr(h))
\end{align}
which can be seen as follows: 
Let $\gfr=\tfr\oplus \bigoplus_{\alpha\in R}\gfr_\alpha$ be the root space decomposition with respect to $\tfr$. 
Then we get 
\[
Z_\gfr(h)=\tfr\oplus \bigoplus_{\alpha \in R, \alpha(h)=0} \gfr_\alpha.  
\]
Since $[\tfr,\gfr_{\alpha}]\neq \{0\}$  and $[h,\gfr_\alpha]=0$ iff $\alpha(h)=0$, the equality (\ref{CenterOfCentralizer}) follows. 
\end{proof}

\begin{lem}
If $x=h$ is semisimple and $t\in \tfr$ with $q(t)=\chi(h)$, then $\mathrm{im}(dq_t)=\mathrm{im}(d\chi_h)$. 
\end{lem}
\begin{proof}
If $x=h+v$ is the Jordan decomposition of $x\in \gfr$, for $h$ semisimple and $v$ nilpotent, then (see \cite{Rich})
\begin{align}\label{CaseOfInterestDerivatives}
\mathrm{rk}(d\chi_x)=\dim C(Z_\gfr(h))+ \mathrm{rk}(d\chi_{1,v}).
\end{align}
Here $\chi_1:\gfr_1:=[Z_{\gfr}(h), Z_{\gfr}(h)] \to \tfr_1/W_1$ is the induced adjoint quotient of the semisimple Lie algebra $\gfr_1$ and $C(Z_{\gfr}(h))$ is the cetner of the centralizer of $h$ in $\gfr$. 
By the previous lemma, it remains to show 
that  $d\chi_0=0$ for any semisimple adjoint quotient $\chi:\gfr\to \tfr/W$. 
This follows from the fact that the degrees $d_i$ of any basis $\hat{\chi}_i$ of $G_{ad}$-invariant polynomials are greater or equal $2$ because $\gfr$ is semisimple (\cite{BourbakiLie4-6}). 
\end{proof}

\begin{prop}\label{p:transversal}
Let $b:U\to \tfr/W$ be a morphism from an open $U\subset \C$ which is transversal to $q:\tfr\to \tfr/W$. Then it is also transversal to $\chi:\gfr\to \tfr/W$ and $\sigma:S\to \tfr/W$. 
\end{prop}
\begin{proof}
Let $x=h+v\in \gfr$ and $t\in \tfr$ such that $\chi(x)=q(t)$. 
The previous corollary together with (\ref{CaseOfInterestDerivatives}) implies that 
\[
r \geq \mathrm{rk}(d\chi_x)\geq \mathrm{rk}(dq_t)\geq r-1.
\]
If $x=h$ is semisimple, then $\chi$ is also transversal to $b$ at $x$ by the previous lemma. 
So we are left with the case $x=h+v$ and $\mathrm{rk}(d\chi_x)=r-1=\mathrm{rk}(dq_t)$. We claim that $v=0$ and $x=h$ must be semisimple which concludes the proof. Without loss of generality we assume again that $h\in \tfr$. 
Since $\dim C(Z_\gfr(h))=\mathrm{rk}(dq_t)=r-1$ by Lemma \ref{RKdq}, it follows that $Z_\gfr(h)=\tfr\oplus \gfr_\alpha \oplus \gfr_{-\alpha}$ for a root $\alpha$ with respect to $\tfr$. Therefore the derived algebra is 
\[
[Z_\gfr(h),Z_\gfr(h)]=\langle h_\alpha, \gfr_{\pm \alpha}\rangle \cong \mathfrak{sl}_2(\C),
\]
where $h_\alpha$ generates the commutator $[\gfr_\alpha,\gfr_{-\alpha}]\subset \tfr$. 
As a consequence, $v$ can be considered as a nilpotent element in $\mathfrak{sl}_2(\C)$ because $v\in \gfr_\alpha\oplus \gfr_{-\alpha}\subset Z_\gfr(h)$. 
By formula (\ref{CaseOfInterestDerivatives}) and $\mathrm{rk}(dq_t)=r-1$ we must have $\mathrm{rk}(d\chi_{1,v})=0$ for the adjoint quotient $\chi_1=\det:\mathfrak{sl}_2(\C)\to \tfr_1/W_1$. 
But $d_A \det=(-2a,-c,-b)$ for $A=aH+bX+cY\in \mathfrak{sl}_2(\C)$ in the standard basis $H,X,Y$ of $\mathfrak{sl}_2(\C)$. 
Hence we must have $v=0$, i.e. $x=h$ is semisimple (and subregular). 
\end{proof}

\section{Calabi-Yau orbifolds over Hitchin bases}\label{s:cy}
\subsection{Hitchin bases and cameral curves}\label{ss:hitchinbase}
This subsection mainly fixes notation and we refer to \cite{DG}, \cite{DP} for more details.
As before, we let $G$ be a simple complex Lie group with Dynkin diagram $\Delta=\Delta(G)$ and $\tfr\subset \gfr=\gfr(\Delta)$ a Cartan subalgebra in $\gfr=\mathrm{Lie}(G)$. 
If $\cu$ is a compact connected Riemann surface of genus $\geq 2$, we denote by $\mathcal{M}_{\mathrm{H}}(\cu,G)$ the moduli space of semistable $G$-Higgs bundles of degree $1\in \pi_1(G)$ and by 
\begin{equation*}
\Hit:\mathcal{M}_{\mathrm{H}}(\cu,G)\to \Bb(\cu, G)
\end{equation*}
the Hitchin map. 
It maps to the Hitchin base 
\begin{equation*}
\Bb\colon=\Bb(\cu, G)=H^0(\cu, \Ub)
\end{equation*}
for the bundle  
$
u\colon\Ub=K_\cu \times_{\C^*} \tfr/W\to \cu
$
of cones where $\C^*$ acts by the standard action on $\tfr/W$. 
Choosing generators $\chi_1,\dots ,\chi_r\in \C[\tfr]^W$ of Weyl group invariants gives $\Ub$ the structure of a vector bundle via the isomorphism
\begin{equation*}
\Ub\cong \bigoplus_{j=1}^r K_\cu^{d_j}
\end{equation*}
for the degrees $d_j=\deg(\chi_j)$. 
In particular, the Hitchin base $\Bb$ inherits a vector space structure. 
By its construction, $\Bb$ only depends on the Lie algebra $\gfr$ of $G$ so that the notation
\begin{equation*}
\Bb(\cu,\Delta):=\Bb(\cu, G)
\end{equation*}
is justified. 
The total space of the vector bundle 
$
\tilde{u}\colon\Ubt=K_\cu\otimes \tfr \to \cu
$
maps to $\Ub$ via the natural quotient map $\bm{q}\colon\Ubt \to \Ub$. 
The fibers of the universal cameral curve 
\begin{equation*}
\bm{p}:\btSigma:=ev^*\Ubt \to \Bb, 
\end{equation*}
for the evaluation map $ev:\cu\times \Bb\to \Ub$, are the cameral curves $\tcu_{b}:=\bm{p}^{-1}(b)$. 
These are smooth over the Zariski-open and dense 
\begin{equation*}
\Bbo:=\{ b\in \Bb~|~b \mbox{ intersects }\mathrm{disc}(\bm{q})\mbox{ transversally} \}
\end{equation*}
and the generic Hitchin fibers $\Hit^{-1}(b)$, $b\in \Bbo$, are generalized Prym varieties defined in terms of $\tcu_b$. 
%
\\
For later reference, we further need the cone
\begin{equation}\label{eq:B/W}
\tilde{\Bb}/W\hookrightarrow \Bb, \quad \tilde{\Bb}:=H^0(\cu, \Ubt),
\end{equation}
where the Weyl group $W$ acts pointwise on sections.
Away from $0\in \Bb$, this is the locus of completely reducible but reduced cameral curves, i.e. 
\begin{equation*}
\tcu_b=\coprod_{w\in W} \tcu_{b,w}, \quad \tcu_{b,w}\cong \cu. 
\end{equation*}
The irreducible components $\tcu_{b,w}$ intersect over the points of the divisor of the section
\begin{equation*}
\prod_{\alpha \in R} \alpha(b)\in H^0(\cu, K_\cu),
\end{equation*}
where $R=R(\Delta)$ is the corresponding root system (which makes sense because $\prod_{\alpha\in R}\alpha \in \C[\tfr]^W$). 

\subsection{Surfaces}
For the following constructions, we fix a spin bundle $L\in \mathrm{Pic}^{g-1}(\cu)$, $L^2=K_{\cu}$ and denote by $\alpha=\alpha_L\in H^1(\cu, \Oo^*)$ its cohomology class.
Moreover, let $\Delta$ be an irreducible Dynkin diagram with associated symmetry group $\Cc=AS(\Delta)$ and $S=S(\Delta)\subset \gfr(\Delta)$ a Slodowy slice. 
Twisting the $\C^*$-spaces $S$ and $\tfr/W$ by $L$, we obtain\footnote{Recall that we usually let $\C^*$ act on $\tfr/W$ by twice the natural weights.}
\begin{equation}\label{eq:Scb}
\Scb:=L\times_{\C^*} S, \quad L\times_{\C^*} \tfr/W \cong \Ub.
\end{equation}
By its $\C^*$-equivariance, the morphism $\sigma:S\to \tfr/W$ and its simultaneous resolution $\tilde{\sigma}:\tilde{S}\to \tfr$ glue to give the following commutative diagram 
\begin{equation}\label{eq:tScb}
\begin{tikzcd}
\tScb \ar[r, "\bm{\psi}"] \ar[d, "\btsigma"'] & \Scb \ar[d, "\bsigma"] \\
\Ubt \ar[r, "\bm{q}"] & \Ub 
\end{tikzcd}
\end{equation}
of $\Cc$-spaces where $\Cc$ acts trivially on $\Ubt$ and $\Ub$. 
However, it takes more care to glue the sections $\hat{\omega}\in \Gamma(\tilde{S},\Omega^2_{\tilde{\sigma}})$ and $\hat{\nu}\in \Gamma(S, K_{\sigma})$. 
\begin{lem}\label{l:bhatnu}
With the notation of (\ref{eq:Scb}) and (\ref{eq:tScb}), the following holds:
\begin{enumerate}[label=\roman*)]
\item 
The section $\hat{\omega}$ glues to a section $\hat{\bomega}\in \Gamma(\tScb,\Omega^2_{\bsigma}\otimes (\tilde{u}\circ \btsigma)^*K_\cu)$
which is $\Cc$-invariant. 
It induces a fiberwise period map 
\begin{equation*}
\bm{\eta}:\Ubt \to \tilde{u}^*\Ubt
\end{equation*}
which coincides with the tautological section $\bm{\tau}\in \Gamma(\Ubt, \tilde{u}^*\Ubt)$. 
\item
Likewise, the section $\hat{\nu}$ glues to a $\Cc$-invariant section $\hat{\bm{\nu}}\in \Gamma(\Scb, K_{\bsigma}\otimes (u\circ \bsigma)^*K_\cu)$ such that 
\begin{equation*}
\bm{\psi}^*\hat{\bm{\nu}}=\hat{\bm{\omega}}
\end{equation*}
under the natural isomorphism $\bm{\psi}^*K_{\bsigma}\cong \Omega^2_{\btsigma}$. 
\end{enumerate}
In particular, the sections $\hat{\bm{\sigma}}$ and $\hat{\bm{\nu}}$ give respective isomorphisms
\begin{equation*}
\Omega^2_{\bsigma}\cong (\tilde{u}\circ \btsigma)^*K_{\cu}^{-1} \quad \mbox{and} \quad K_{\bsigma} \cong (u\circ \bsigma)^*K_{\cu}^{-1}. 
\end{equation*}
\end{lem}
\begin{proof}
To construct the section $\hat{\bm{\omega}}$, we need the gluing data for the sheaf $\Omega^2_{\btsigma}$. 
Let $\alpha=\alpha_L$ be the cocycle corresponding to $L$ and denote by $(\alpha_{ij})$ a \v{C}ech representative for a fixed open covering $(D_{ij})$ of $\cu$. 
In particular, $(\beta_{ij})=(\alpha_{ij}^2)$ is a cocycle for $K_\cu$. 
Trivializing $\tScb$ over each $D_{i}$ gives rise to the commutative diagram
\begin{equation*}
	\begin{tikzcd}
	& \tScb_{ij} \arrow[r, "\psi_i"]  & D_{ij}\times \tilde{S} \ar[r, "\mathrm{id}\times \tilde{\sigma}"] & D_{ij}\times \tfr  \\
	\Ubt \ar[ru, leftarrow, "\tilde{\bsigma}_{ij}"] \ar[rd, leftarrow, "\tilde{\bsigma}_{ij}"'] 
	\\
	& \tScb_{ij} \arrow[r, "\psi_j"] \arrow[uu, equal] & D_{ij}\times \tilde{S} \arrow[uu, "g_{ij}"'] \ar[r, "\mathrm{id}\times \tilde{\sigma}"]& D_{ij}\times \tfr  \ar[uu, "h_{ij}"'] 
	\end{tikzcd}
	\end{equation*}
over $D_{ij}=D_i\cap D_j$ with
	\begin{align*}
	&g_{ij}(x,s)=(x, \alpha_{ij}(x)\cdot s)=(x, \mu(\alpha_{ij}(x),s)),\\
	&h_{ij}(x,t)=(x,  \alpha_{ij}(x)\cdot t)=(x, \beta_{ij}(x)t). 
	\end{align*}
for $(x,s)\in D_{ij}\times \tilde{S}$ and the action map $\mu:\C^*\times \tilde{S}\to \tilde{S}$. 
On each $D_i\times \tilde{S}$ we have the sheaves $\mathcal{E}_i:=\Omega^2_{\mathrm{id}\times \tilde{\sigma}}\cong \mathrm{pr}_{2,i}^*\Omega^2_{\tilde{\sigma}}$ together with the sections $\mathrm{pr}_{2,i}^*\hat{\omega}$. 
Clearly, $\mathcal{E}_i$ and $\mathcal{E}_j$ are canonically isomorphic over $D_{ij}$. 
Now $\Omega^2_{\btsigma}$ is glued from\footnote{Note that $\psi_i^*\mathcal{E}_i\cong \Omega^2_{\tilde{\bsigma}|\tSc_i}$.}
 $\psi_i^*\mathcal{E}_i$ on $D_{ij}$ via the isomorphisms
	\begin{equation*}
	\begin{tikzcd}
	\varphi_{ij}:=\psi_i^*dg_{ji}^t:  \psi_j^*\mathcal{E}_j=\psi_i^*g_{ji}^*\mathcal{E}_j \arrow[r] & \psi_i^*\mathcal{E}_i
	\end{tikzcd}
	\end{equation*}
over $D_{ij}$. 
Here we denote by $dg_{ji}^t:g_{ij}^*\mathcal{E}_j\to \mathcal{E}_i=\mathcal{E}_j$ the natural morphism (over $D_{ij}$). 
Observe that we have $\varphi_{ij}\circ \varphi_{jk}=\varphi_{ik}$  and $(\varphi_{ij})$ is the gluing (or descent) datum for $\Omega^2_{\btsigma}$.
Indeed, we can write this composition as 
	\begin{equation*}
	\begin{tikzcd}
	\psi_k^*\mathcal{E}_k\ar[rr, "\varphi_{kj}"] \ar[d, "\cong"]  && \psi_j^*\mathcal{E}_j \ar[rr, "\varphi_{ji}"]  \ar[d, "\cong"] && \psi_i^*\mathcal{E}_i \ar[d, "\cong"] \\
	\psi_i^*(g_{ji}^*) g_{kj}^*\mathcal{E}_j \ar[rr, "\psi_i^* g_{ji}^* dg_{kj}^t"] && \psi^*_i g_{ji}^*\mathcal{E}_j \ar[rr, "\psi_i^*dg_{ji}^t"] && \psi_i^*\mathcal{E}_j.
	\end{tikzcd}
	\end{equation*}
The lower line is $\varphi_{ik}$ by the chain rule, showing the cocycle condition for $(\varphi_{ij})$. 
	\begin{claim}
		Define the local sections 
		\begin{equation*}
		\hat{\omega}_i:=\psi_i^*\mathrm{pr}_{2,i}^*\hat{\omega}\in \Gamma(\mathcal{S}_i,\psi_i^*\mathcal{E}_i).
		\end{equation*}
		On the overlaps $\tSc_{ij}$, they transform as follows
		\begin{equation}\label{TrafoRuleSections}
		\varphi_{ij}(\hat{\omega}_j)=((\mathrm{pr}_{1,i}\circ \psi_i)^*\beta_{ji}) \, \hat{\omega}_i. 
		\end{equation}
	\end{claim}
Before we prove this claim, let us see how it yields the desired section. 
Observe that $(\mathrm{pr}_{1,i}\circ\psi_i)^* \beta_{ji}$ is a cocycle for $(\tilde{u}\circ \btsigma)^*K_{\Sigma}^{-1}$. 
Hence in order to obtain a well-defined global section on $\tilde{\mathcal{S}}$, we have to tensor with $(\tilde{u}\circ \btsigma)^*K_{\Sigma}$. 
More precisely, let $\zeta_i\in \Gamma(D_i,K_\Sigma)$ be the local frames of $K_\cu$ over $D_i$ so that $\zeta_i=\beta_{ij}~\zeta_j$ on $D_{ij}$. 
Letting $\hat{\zeta}_i:=\psi_i^*\mathrm{pr}_{1,i}^*\zeta_i$, we see that the local sections
	\begin{equation*}
	\hat{\omega}_i\otimes \hat{\zeta}_i\in\Gamma(\tScb_i,\Omega^2_{\btsigma}\otimes (\tilde{u}\circ \btsigma)^*K_\Sigma)
	\end{equation*}
glue to give the global section $\bm{\hat{\omega}}\in \Gamma(\tScb,\Omega^2_{\btsigma}\otimes (\tilde{u}\circ \btsigma)^*K_\Sigma)$ as claimed. 
\\ 
\\
We still have to give a proof of (\ref{TrafoRuleSections}). 
To simplify notation, we drop the subscript $ij$ if not necessary and only write $g:D\times \tilde{S}\to D\times \tilde{S}$ etc. 
Then the second component of $dg:TD\oplus T\tilde{S}\to TD\oplus T\tilde{S}$ at $(x,s)\in D\times \tilde{S}$ is given by  
\begin{equation}\label{SecondComponent}
	d\mu_{\alpha(x)\cdot s}(d\alpha_x(v),w)=d\mu_{\alpha(x)\cdot s}(d\alpha_x(v),0)+d\mu_{\alpha(x)\cdot s}(0, w). 
\end{equation}
Note that $d\mu_{\alpha(x)\cdot s}(0,w)=d(\mu_{\alpha(x)})_s(w)$ where $\mu_{\alpha(x)}=\mu(\alpha(x),-)$.
Now let $p\in \tSc_{ij}$ and $\psi_i(p)=(x,s)\in D_{ij}\times \tilde{S}$.
Then we clearly have  
	\begin{equation*}
	\ker d_{(x,s)}(\mathrm{id}\times \tilde{\sigma})=0\oplus \ker d_s\tilde{\sigma}\subset T_xD_{ij}\oplus T_s\tilde{S}. 
	\end{equation*}
In particular, the first summand in (\ref{SecondComponent}) plays no role for our discussion. 
	For $w_k\in \ker d_s\tilde{\sigma}$ ($k=1,2$) one computes
	\begin{align*}
	&\quad~\varphi_{ij}(\hat{\omega}_j)_p\Big((0,w_1),(0,w_2)\Big)\\
	&=\mathrm{pr}_{2,j}^*\hat{\omega}_{g_{ji}(x,s)}\circ dg_{ji, (x,s)}\Big((0,w_1),(0,w_2)\Big) & (\psi_j\circ \psi_i^{-1}=g_{ji}) \\
	&=\beta_{ji}(x) \, \hat{\omega}_s(w_1,w_2) & (\C^*\text{-equivariance and }(\ref{SecondComponent}))\\ 
	&=(\mathrm{pr}_{1,i}\circ \psi_i)^*\beta_{ji}(p) \, (\hat{\omega}_i)_p\Big((0,w_1),(0,w_2)\Big). 
	\end{align*}
Here we have used the $\C^*$-equivariance of the relative form $\hat{\omega}$ showing (\ref{TrafoRuleSections}). 
	\\
To construct the period map $\bm{\eta}$ out of $\hat{\bomega}$, observe that it induces the morphism
	\begin{equation*}
	(x,t)\mapsto P_{\tilde{S}}(t)\otimes ((x,t),\zeta_i(x))
	\end{equation*}
in each trivialization $\Ubt_{|D_i} \cong D_i\times \tfr$ where $P_{\tilde{S}}$ is as in (\ref{eq:H2per}). 
These morphisms glue to give the morphism 
	\begin{equation*}
	\bm{\eta}: \Ubt\to \tfr\otimes\tilde{u}^*K_\Sigma
	\end{equation*}
which coincides with the tautological section $\bm{\tau}\in H^0(\Ubt,\tilde{u}^*\Ubt)$ by construction. 
\\
The existence of the global section $\bm{\hat{\nu}}\in H^0(\Scb,K_{\bsigma}\otimes (u\circ \bsigma)^*K_\Sigma)$ works similarly:
There exists a section $\bm{\hat{\nu}}^{reg}\in \Gamma(\Scb^{reg}, K_{\bsigma})$, which is constructed as $\hat{\bm{\omega}}$ by replacing $\hat{\omega}$ with $\hat{\nu}$. 
Here $\Scb^{reg}\subset \Scb$ is the locus which is glued from $S^{reg}\subset S$. 
Using a codimension argument as in the proof of Proposition \ref{p:C*equivariant}, we see that it uniquely extends to a section $\bm{\hat{\nu}}\in \Gamma(\Scb, K_{\bsigma})$. 
It satisfies $\bm{\psi}^*\bm{\hat{\nu}}=\bm{\hat{\omega}}$ under the isomorphism
$\bm{\psi}^*K_{\bsigma}\cong K_{\btsigma}$ by construction, since this holds for the corresponding local sections. 
\end{proof}

\subsection{Calabi-Yau threefolds}
The family $\bsigma: \Scb\to \Ub$ of surfaces over $\Ub$ pulls back to the product $\cu\times \Bb$, $\Bb=\Bb(\cu, \Delta)$ via the evaluation map $\ev:\cu \times \Bb(\cu,\Delta)\to \Ubt$. 
Projecting to the second factor yields the family 
\begin{equation*}
\begin{tikzcd}
\X_L(\Delta) \ar[rr, bend left, "\pib_L"] \ar[r, "\pib_{1,L}"] & \cu \times \Bb \ar[r, "\pib_{2,L}"] & \Bb
\end{tikzcd}
\end{equation*}
of threefolds with a $\Cc$-action over the Hitchin base $\Bb$. 
Analogously, pulling back $\tScb$ along the natural map $\btcu\to \Ubt$ gives the family 
\begin{equation*}
\begin{tikzcd}
\tilde{\X}_L(\Delta) \ar[rr, bend left, "\tilde{\pib}_L"] \ar[r, "\tilde{\pib}_{1,L}"] & \btcu \ar[r, "\tilde{\pib}_{2,L}"] & \Bb 
\end{tikzcd}
\end{equation*}
of threefolds with $\Cc$-action over $\Bb$. 
\begin{thm}\label{thm:cy3s}
Let $\Delta$ be an irreducible Dynkin diagram and $L$ be a spin bundle of $\cu$. 
Then 
\begin{equation*}
\begin{tikzcd}
\X_L(\Delta)\ar[r] & \Bb(\cu, \Delta) & \ar[l] \tilde{\X}_L(\Delta)
\end{tikzcd}
\end{equation*}
are algebraic families of quasi-projective Gorenstein threefolds with $\Cc$-trivial canonical class. 
They are smooth over the locus $\Bbo(\cu,\Delta)\subset \Bb(\cu, \Delta)$ of smooth cameral curves.
\end{thm}
\begin{proof}
The projections $\Ub\to \cu$ and $\Scb \to \cu$ are quasi-projective morphisms because they are (affine) bundles over the projective curve $\cu$. 
Therefore the morphism $\bsigma: \Scb \to \Ub$ which factorizes as 
\begin{equation*}
\begin{tikzcd}
\Scb \ar[d, "\bsigma"] \ar[dr] & \\
\Ub \ar[r ] & \cu
\end{tikzcd}
\end{equation*}
is quasi-projective as well. 
Hence $\pib:\X \to \Bb$ is quasi-projective. 
Using the fact that $\bsigma: \Scb\to \Ub$ is a Gorenstein morphism, a similar argument shows that each $X_b=\pib^{-1}(b)$, $b\in \Bb$, is Gorenstein. 
\\
To see the statement about the canonical class, let $j:\X\to \Scb$ be the natural morphism obtained from base change. 
Then the section $\hat{\bm{\nu}}$ of Lemma \ref{l:bhatnu} pulls back to (dropping $L$ from the notation)
\begin{equation}
	\bm{s}:=j^*\hat{\bm{\nu}}\in H^0(\X, K_{\pib_1} \otimes (\mathrm{pr}_1\circ \pib_1)^*K_\Sigma).
\end{equation}
Base change and the adjunction formula yields the isomorphism
\begin{equation*}
		(K_{\pib_1}\otimes (\mathrm{pr}_1\circ \pib_1)^*K_\cu)_{|X_b}\cong K_{\pi_b}\otimes \pi_b^*K_\cu \cong K_{X_b}. 
	\end{equation*}
Hence each restriction $s_b:=\bm{s}_{|X_b}$ is a nowhere vanishing section of the locally free sheaf $X_b$. 
Since the $\Cc$-action on $X_b$ is pulled back from the $\Cc$-action on $\Scb$, $s_b$ is $\Cc$-invariant by Lemma \ref{l:bhatnu}. 
Finally, the statement about smoothness follows from Proposition \ref{p:transversal}.
\\
The proof for $\tilde{\X}_L(\Delta)$ works in complete analogy by replacing $\bm{s}$ with
\begin{equation}
\tilde{\bm{s}}:=\tilde{j}^*\hat{\bomega}\in H^0(\tilde{\X}, K_{\tilde{\pib}_1}\otimes (\mathrm{pr}_1\circ \bm{p}_1)^*K_\cu )
\end{equation}
for the natural morphism $\tilde{j}:\tilde{\X}\to \tScb$.
\end{proof}
We close this subsection with the deformation-theoretic meaning of the Hitchin base $\Bb=\Bb(\cu, \Delta)$ for the family $\X\to \Bb$ (see \cite{DDD} for the $\mathrm{A}_1$-case). 
To this end, observe that the central fiber 
\begin{equation*}
X_0=L\times_{\C^*} S_{\bar{0}}
\end{equation*}
has a curve of $\Delta$-singularities along $\cu \hookrightarrow X_0$. 
It is easily seen from the construction that $X_b$, $b\neq 0$, only has isolated singularities. 
In particular, the deformation $\X\to \Bb$ of $X_0$ is not a locally trivial one. 
More precisely, we have:
\begin{thm}\label{thm:defX0}
Let $X=X_0$ be the central fiber of a family $\X\to \Bb$ of quasi-projective Gorenstein Calabi-Yau with $\Cc$-trivial canonical class. 
Then 
\begin{equation*}
\mathrm{Ext}^1(\Omega_X^1,\Oo_X)/H^1(X,T_X) \cong \Bb, 
\end{equation*}
i.e. the space of $\Cc$-deformations of $X$ modulo locally trivial $\Cc$-deformations is isomorphic to the corresponding Hitchin base. 
\end{thm}
\begin{proof}
We first consider the case $\Cc=1$ and set $r:=\mathrm{rk}(\gfr)$. 
Since $\pi:X\to \cu$ is affine, the Leray spectral sequence implies that $H^2(X,T_X)=0$. 
Therefore the local-to-global spectral sequence implies 
\begin{equation*}
\mathrm{Ext}^1(\Omega_X^1, \Oo_X)/H^1(X,T_X)\cong H^0(X, \ExtSO^1(\Omega_X^1,\Oo_X)). 
\end{equation*}
If $k:\cu\hookrightarrow X$ is the closed imbedding, then we claim
\begin{equation}\label{eq:ext1}
\mathscr{E}:=\ExtSO^1(\Omega_X^1,\Oo_X)\cong k_*\Ub. 
\end{equation}
As a first step, we realize $X$ as a regular imbedding.
Let $\pi_{\Scb}:\Scb\to \cu$ be the natural projection and $\pi_{\Scb}^*{\Ub} \to \Scb$ the induced vector bundle. 
Then there exists a unique section $\tau:\Scb\to \pi^*_{\Scb} \Ub$ such that $pr_{\Ub}\circ \tau=\bsigma$ for the natural projection $pr_{\Ub}:\pi^*_{\Scb}\Ub \to \Ub$. 
In a trivialization of $\Scb_{|D}\cong D\times S$ for an affine $D\subset \cu$, the section is given by 
\begin{equation*}
D\times S \to (D\times S) \times \tfr/W, \quad (x,s)\mapsto ((x,s),\sigma(s)). 
\end{equation*}
Therefore the vanishing locus $Z(\tau)\subset \Scb$ coincides with $i:X\hookrightarrow \Scb$ and is a regular imbedding. 
The latter implies that the normal sheaf satisifes
\begin{equation}\label{eq:normalsheaf}
\mathcal{N}_{X/\Scb}\cong \pi^*_{\Scb}\Ub_{|X}, 
\end{equation}
cf. \cite{FultonLang}, Chapter IV. 
In order to compute $\mathscr{E}$, let $\mathcal{J}$ be the ideal sheaf defining $X\subset \Scb$. 
Since the latter is a regular imbedding, we have the locally free resolution 
\begin{equation*}
\begin{tikzcd}
0 \ar[r] & \mathcal{J}/\mathcal{J}^2 \ar[r] & i^*\Omega^1_{\Sc} \ar[r] & 0 
\end{tikzcd}
\end{equation*}
of $\Omega_X^1$. 
Dualizing gives the sequence
\begin{equation*}
\begin{tikzcd}
0 \ar[r] & \HomS_{\Oo_X}(i^*\Omega^1_{\Scb},\Oo_X) \ar[r, "\varphi"] & \mathcal{N}_{X/\Scb} \ar[r] & 0
\end{tikzcd}
\end{equation*}
so that $\mathrm{coker}(\varphi)=\ExtSO^1(\Omega_X^1,\Oo_X)$ which is supported on $\cu \subset X$. 
In order to prove (\ref{eq:ext1}) it therefore suffices to show that $\mathscr{E}$ has constant rank $r$ along $\cu$ in light of (\ref{eq:normalsheaf}).
\\
Let $U:=D\times S_{\bar{0}}\subset X$ and $D\subset \cu$ be an affine open in $X$ and $\cu$ respectively. 
Hence 
\begin{equation*}
U=\mathrm{Spec}(B), \quad B=\C[t,x,y,z]/\langle f(x,y,z) \rangle,\quad D=\mathrm{Spec}(C), \quad C=\C[t,x,y,z],
\end{equation*}
where $f$ defines the $\Delta$-singularity $S_{\bar{0}}$. 
The sheaf $\mathscr{E}$ corresponds to the $B$-module
\begin{equation*}
\mathrm{Ext}^1_B(\Omega^1_B, B) \cong \C[t,x,y,z]/\langle f, \partial f \rangle.
\end{equation*}
As a $C$-module, it is isomorphic to $C[t]^r$. 
This concludes the proof for $\Cc=1$. 
If $\Cc\neq 1$, the same proof goes through if we work with $\Cc$-invariants and recall that $\Bb_h^\Cc=\Bb$. 
\end{proof}

\subsection{Relation between \texorpdfstring{$\mathcal{X}_L(\Delta)$}{ } and \texorpdfstring{$\mathcal{X}_L(\Delta_h)$}{}}
We next compare the two families 
\begin{gather*}
\X=\X_L(\Delta) \to \Bb=\Bb(\cu, \Delta)\\
\X_h=\X_L(\Delta_h)\to \Bb_h=\Bb(\cu, \Delta_h)
\end{gather*}
in case $\Delta\neq \Delta_h$. 
As a first step, we construct a non-trivial $AS(\Delta)$-action on $\X_h$. \\
Let $S_h:=S(\Delta_h)=x+\ker(\mathrm{ad}(y))\subset \gfr_h:=\gfr(\Delta_h)$ be a Slodowy slice. 
Then the group 
\begin{equation}\label{eq:CA}
CA(x,h):=\{ \phi \in \Aut(\gfr_h) ~|~\phi(x)=x, \phi(h)=h \}
\end{equation}
acts on $S_h$. 
There is a subgroup $\CcA\subset CA(x,h)$ which is isomorphic to $AS(\Delta)$ (see \cite{Slo}, 7.5). 
Of course, the definition of $\CcA$ and $\Cc$ makes sense for $S$ and $S_h$ respectively. 
Then $\CcA\cong \Cc$ in the former and $\Cc=1$ in the latter case. 
\begin{rem}
Even though $\Cc\cong AS(\Delta) \cong \CcA$, we often write $\Cc$ and $\CcA$ to emphasize how the $AS(\Delta)$-action is realized. 
To illustrate the relation between $S$ and $S_h$ together with the corresponding $AS(\Delta)$-actions, we give a worked out example in Appendix \ref{a:ex}. 
\end{rem}
The $\CcA$-action on $S_h$ induces a $\CcA$-action on $\tfr_h/W_h$ such that $\sigma_h:S_h\to \tfr_h/W_h$ is equivariant. 
Since the $\CcA$-action commutes with the $\C^*$-action on $S_h$, we obtain a $\CcA$-action on $\X_h$. 
Lemma \ref{l:bhatnu} and Theorem \ref{thm:cy3s} analogously hold for the $\CcA$-action (cf. \cite{Beck-thesis} for details). 
However, $\CcA$ acts non-trivially on the base so that a general member $X_{h,b}$ does not have $\CcA$-trivial canonical bundle.
\begin{cor}\label{cor:comp}
With the previous notation, 
\begin{equation}\label{eq:incl}
\Bb\cong \Bb_h^{\CcA}\hookrightarrow \Bb_h
\end{equation}
but $\Bb^\circ \cap \Bb_h^\circ =\emptyset$. 
The family $\X\to \Bb$ is the restriction of the family $\X_h\to \Bb_h$ under the inclusion (\ref{eq:incl}) and the $AS(\Delta)$-action on $\X$ is induced by the $AS(\Delta)$-action on $\X_h$. 
In particular, $\X_h$ is smooth over a Zariski open and dense subset containing $\Bbo \sqcup \Bbo_h\subset \Bb_h$. 
\end{cor}
\begin{proof}
It is not difficult to see that $\tfr/W\cong (\tfr_h/W_h)^{\CcA}$ which gives the inclusion $i:\tfr/W \hookrightarrow \tfr_h/W_h$ and hence (\ref{eq:incl}). 
To see that $\Bb^\circ \cap \Bb_h^\circ=\emptyset$, let $\alpha\in \Delta$ be a long root.
Under folding, it corresponds to an $AS(\Delta)$-orbit $O(\beta)$ of length $\geq 2$ for some $\beta\in \Delta_h$. 
If $\tfr_\alpha\subset \tfr$ is the fixed point locus of $s_\alpha\in W$ (and analogously for $\tfr_h$), then 
\begin{equation}\label{eq:talpha}
\tfr_\alpha=\bigcap_{\beta'\in O(\beta)} \tfr_{h,\beta'}\subset \tfr_h
\end{equation}
under the inclusion $\tfr\subset \tfr_h$. 
If $b\in \Bbo$, then $b$ necessarily maps to the smooth locus $\mbox{disc}(\bm{q})^{sm}$ of the discriminant of $\bm{q}:\Ubt\to \Ub$. 
Since $\tcu_b$ necessarily intersects $K_\cu\otimes \tfr_{\alpha}\subset \Ubt$ and because of (\ref{eq:talpha}), $b$ cannot map to $\mbox{disc}(\bm{q}_h)^{sm}$ 
under the inclusion $\Ub\subset \Ub_h$.  
Hence $\Bbo\cap \Bbo_h=\emptyset$. 
\\
Slodowy has proven (Chapter 8.8 in \cite{Slo}) that $i^*S_h\cong S$ over $\tfr/W$ as $\C^*$-deformations of the $\Delta$-singularity over $\bar{0}\in \tfr/W$. 
By the $\C^*$-equivariance, this isomorphism induces the isomorphism $\X\cong \X_h$ over $\Bb\subset \Bb_h$. 
\end{proof}

\subsection{Equations}
To connect to similar constructions in the physical mathematics literature (\cite{DDD}, \cite{Tbranes}, \cite{Tbranes2}), we give explicit equations of the families $\X\to \Bb$ in appropriate bundles over $\cu$. 
\\
If we identify $S=\cong \C^{r+2}$ appropriately, then the weights of the $\C^*$-action are given by $w_1=2d_1,\dots, w_r=2d_r, w_{r+1},w_{r+2},w_{r+3}$ for the degrees $d_j=\mathrm{deg}(\chi_j)$ as in Section \ref{ss:hitchinbase} and weight $w_{r+1},w_{r+2},w_{r+3}$ as in Table \ref{ADEsings} in Appendix \ref{SectFolding}.
It follows that
\begin{equation*}
\Scb=L\times_{\C^*} S \cong \bigoplus_{j=1}^{r} K^{d_j} \oplus \bigoplus_{j=1}^3 L^{w_{r+k}}.
\end{equation*}
Let $f\in \C[x,y,z]$ define the $\Delta$-singularity and let $g_1,\dots, g_r\in \C[x,y,z]^{\Cc}$ be representatives of generators of the $\Cc$-invariant Jacobi ring $\C[x,y,z]^{\Cc}/(f,\partial f)$.
Then the family $\X\to \Bb$ reads as
\begin{equation}\label{eq:explicit}
\X=\{ (x,y,z,\underline{b})~|~f(x,y,z)+\sum_{i=1}^r b_i g_i(x,y,z)=0 \in \mathrm{tot}(K^d) \}
\end{equation}
in $\mathrm{tot}(\bigoplus_{i=1}^3 L^{w_{r+k}})\times \Bb$ with the obvious projection and where $d$ is again as in Table \ref{ADEsings}.

\begin{ex}[$\Delta=\mathrm{B}_2$]\label{ex:B2}
Let us give at least one simple example (also compare with Appendix \ref{a:ex}).
In this case $f(x,y,z)=x^4-yz$ and $\Cc=\Z/2\Z$ acts as $(x,y,z)\mapsto (-x,z,y)$. 
Therefore (\ref{eq:explicit}) reads as 
\begin{equation*}
\X=\{ (x,y,z, (b_1,b_2))~|~x^4-yz+b_1 x^2+ b_2=0 \in \mathrm{tot}(K^4) \}
\end{equation*}
in $\mathrm{tot}(K\oplus K^2\oplus K^2)\times \Bb$ with $\Bb=H^0(\cu, K^2)\oplus H^0(\cu, K^4)$. 
\end{ex}

%

\subsection{Simultaneous resolutions}
The families $\pib:\X\to \Bb$ and $\tilde{\pib}:\tilde{\X}\to \Bb$ share many features of the Slodowy slice $S$. 
It seems unlikely though that our methods provide an explicit description of a simultaneous resolution for these families. 
However, we succeed over the locus $\tilde{\Bb}/W\subset \Bb$ (see (\ref{eq:B/W})) which is in the singular locus of $\pib$.
\\
To see this, let $\tilde{ev}:\cu\times \tilde{\Bb}\to \tilde{\Bb}$ be the evalution map which gives rise to the family 
\begin{equation*}
\begin{tikzcd}
\tilde{ev}^*\tilde{\Scb}=\hat{\X}_L \ar[rr, bend left, "\hat{\pib}_L"] \ar[r, "\hat{\pib}_{1,L}"] & \cu\times \tilde{\Bb} \ar[r, "\hat{\pib}_{2,L}"] & \tilde{\Bb}
\end{tikzcd}
\end{equation*}
of threefolds with a non-trivial $\Cc$-action.
With the methods of the proof of Theorem \ref{thm:cy3s}, one proves that $\hat{\pib}_L:\hat{\X}_L\to \tilde{\Bb}$ is a \emph{smooth} family of quasi-projective threefolds with $\Cc$-trivial canonical class.
\begin{prop}\label{p:simresol}
The family $\hat{\pib}:\hat{\X}\to \tilde{\Bb}$ is a simultaneous $\Cc$-resolution of $p^*\X$ where $p:\tilde{\Bb}\to \tilde{\Bb}/W\subset \Bb$ is the natural projection. 
It descends to a simultaneous \emph{small} $\Cc$-resolution over $\tilde{\Bb}/W-\{0\} \subset \Bb$. 
\end{prop}
\begin{proof}
Using the fiber product property, we obtain a unique  $\Cc$-equivariant morphism $\hat{\X}\to p^*\X$ over $\cu\times \tilde{\Bb}$ which makes the obvious diagrams commute. 
\\
For the second claim, observe that the $\Cc$-family 
$
\hat{\X}/W\to \tilde{\Bb}/W
$
is still smooth over $\Bb^*:=\tilde{\Bb}/W-\{0\}\subset \Bb$.
Then the morphisms $\hat{\X}\to p^*\X$ descends to the $\Cc$-morphism
\begin{equation*}
\hat{\X}/W\to p^*\X/W\cong \X
\end{equation*}
over $\Bb^*$. 
It remains to show that $(\hat{X}/W)_{b}\to X_b$ is a small resolution for $b\in \Bb^*$.
Let $\tilde{b}$ with $p(\tilde{b})=b$. 
Then $X_b$ is only singular at the singularities of the fibers $p_b^{-1}(x)$ for
\begin{equation*}
x\in \mathrm{Div}(\prod_{\alpha\in R} \alpha(b) )
\end{equation*}
and the projection $p_b:X_b\to \cu$. 
By assumption $\prod_{\alpha\in R} \alpha (b)\neq 0$, i.e. the singularities are isolated, so that $(\hat{X}/W)_b\to X_b$ is a small resolution. 
\end{proof}
\begin{rem}
The last statement is false if we include $0\in \Bb$: 
Then $X_0\cong L\times_{\C^*} Y$ for the $\Delta$-singularity $Y$ and $\hat{X}_0\cong L\times_{\C^*} \hat{Y}$ for its minimal resolution $\hat{Y}\to Y$. 
Hence it is not a small resolution. 
\end{rem}
\begin{ex}
The singularities of $X_b$, $b\in \Bb^*$, are precisely the Gorenstein threefold singularities studied in \cite{KatzMorr}. 
The simplest (local) example is 
\begin{equation*}
\begin{tikzcd}
&\hat{X}\ar[d]   =\{ (x,y,z,t,[u:v])\in \C^4\times \mathbb{P}^1~|~x^2+y^2+z^2-t^2=0~,~xv=u(z+t)\}
\\
&X   =\{ (x,y,z,s)\in \C^4~|~x^2+y^2+z^2-s=0 \}
\end{tikzcd}
\end{equation*}
for the small resolution is $(x,y,z,t,[u:v])\mapsto (x,y,z,t^2)$. 
\end{ex}

\section{Calabi-Yau orbifolds and Hitchin systems}\label{s:mhs}
Let $\X\to \Bb$ be a family of quasi-projective Calabi-Yau threefolds with $\Cc$-trivial canonical class as in the previous section.
For each $b\in \Bb$, we denote by $[X_b/\Cc]$ the corresponding quotient stack and refer to it as a Calabi-Yau orbifold.  
By construction, they fit into the family
$
\pib_{\Cc}^\circ:[\X/\Cc]\to \Bb
$
of Calabi-Yau orbifolds. 
Before we study the \emph{integral} (equivariant) cohomology groups
\begin{equation*}
H^3([X_b/\Cc],\Z)=H^3_{\Cc}(X_b,\Z),\quad b\in \Bbo, 
\end{equation*}
we need a general result. 
\subsection{MHS on equivariant cohomology}\label{ss:mhseq}
Let $G$ be a finite group acting on a locally compact topological space $X$.
We replace the orbifold stack $[X/G]$ by the simplicial space
\begin{equation}
[X/G]_\bullet=((G^{p+1}\times X)/G)_{p\geq 0}
\end{equation}
where $G$ acts on $G^{p+1}\times X$ via 
\begin{equation*}
g\cdot (g_0,\dots, g_p, x)=(g_0 g^{-1},\dots, g_p g^{-1}, g\cdot x).
\end{equation*}
The simplicial structure maps of $[X/G]_\bullet$ are induced by the standard simplicial structure on $G^{\bullet+1}$. 
The importance of $[X/G]_\bullet$ is that it is a simplicial model for $X_G:=X \times_G EG$ which defines equivariant cohomology 
\begin{equation*}
H_G^k(X,R)=H^k(X_G,R),\quad R=\Z,\Q.
\end{equation*}
More precisely, we have 
\begin{equation}\label{eq:eqcohorb}
H^k([X/G]_\bullet, R) \cong H^k([X/G],R)= H_G^k(X,R), \quad R=\Z,\Q. 
\end{equation}
%
An important tool to compute the left-hand side is the spectral sequence 
\begin{equation}\label{EqCohSS}
E_1^{pq}=H^q([X/G]_p,R) \Rightarrow E_\infty^k=H^k([X/G]_\bullet, R)=H_G^k(X,R),\quad R=\Z,\Q.
\end{equation}
By \cite{DeligneIII}, this spectral sequence is a spectral sequence of MHS if $X$ is a complex algebraic variety (considered in the analytic topology) on which $G$ acts algebraically.  
In particular, the equivariant cohomology groups $H^k_G(X,R)$ carry natural MHS. \\
The Leray (or Serre) spectral sequence for the fibration $X_G\to BG$ implies that $H_G^k(X,\Q)\cong H^k(X,\Q)^G$ as abelian groups. 
Of course, both sides carry MHS and the next lemma shows that they agree.  
This is well-known to experts\footnote{We kindly acknowledge the help of Donu Arapura via MathOverflow.} but for lack of a reference, we give a proof here for completeness.
\begin{lem}\label{lem:EqCohInv}
Let $G$ be a finite group acting on a complex algebraic variety $X$. 
Then the spectral sequence (\ref{EqCohSS}) degenerates on the $E_2$-page and yields an isomorphism
\begin{equation*}
H_G^k(X,\Q)\cong H^k(X,\Q)^G
\end{equation*}
of $\Q$-MHS for each $k\geq 0$. 
\end{lem}
\begin{proof}
We express $E_1^{pq}=H^q((G^{p+1}\times X)/G,\Z)$ of (\ref{EqCohSS}) as follows:
\begin{align*}
E_1^{pq}=& H^q(G^{p+1}\times X,\Z)^G & (\mbox{freeness of action}) \\
=& \Hom_{\Z[G]} (\Z, H^q(G^{p+1}\times X,\Z)) & \\ 
=& \Hom_{\Z[G]} (\Z, H^0(G^{p+1},\Z)\otimes_{\Z} H^{q}(X,\Z) ) & (\mbox{K\"unneth formula}) \\ 
=& \Hom_{\Z[G]} (\Z, \Hom_{\Z} (H_0(G^{p+1},\Z), H^q(X,\Z)))) & (\mbox{freeness of }H_0(G^{p+1},\Z)) \\
=& \Hom_{\Z}(\Z\otimes_{\Z[G]} H_0(G^{p+1},\Z), H^q(X,\Z)) & (\mbox{adjunction}) \\
=& \Hom_{\Z[G]}(H_0(G^{p+1},\Z),H^q(X,\Z)).
\end{align*}
In the last line we have used the adjunction between the trivial module functor and $-\otimes_{\Z} \Z[G]$ (\cite{Weibel}). 
Since the complex $H_0(G^{\bullet+1}, \Z)$ coincides with the bar resolution $B_\bullet \to \Z$, we see that 
\begin{equation*}
E_1^{\bullet q}=\Hom_{\Z[G]}(B_\bullet, H^q(X,\Q))
\end{equation*}
as complexes for each $q\geq 0$. 
In particular, $E_1^{\bullet q}$ computes group cohomology $H^k(G, H^q(X,\Z))$ for each $q\geq 0$ so that $E_2^{pq}\cong H^p(G,H^q(X,\Q))$. 
But $G$ is finite and we work over $\Q$ so that $H^p(G,H^q(X,\Q))=0$ for all $p\geq 1$. 
Hence the spectral sequence (\ref{EqCohSS}) of MHS degenerates on the $E_2$-page to give isomorphisms
\begin{equation*}
H^k(X,\Q)^G\cong H^k([X/G]_\bullet,\Q)=H_G^k(X,\Q)
\end{equation*}
of MHS. 
\end{proof}

\subsection{CY orbifolds and Hitchin systems}
We begin with a general result on the Leray spectral sequence for equivariant maps. 
\begin{lem}\label{lem:leraygroup}
Let $G$ be a discrete group acting on topological spaces $X,Y,Z$. 
Further let $f:X\to Y$, $h:Y\to Z$ be $G$-equivariant morphisms. 
Then the Leray spectral sequence 
\begin{equation}\label{LeraySS}
R^pg_* R^qf_*A_X \Rightarrow R^{p+q}(g\circ f)_*A_X
\end{equation}
for any constant sheaf $A_X$ of an abelian group $A$ on $X$ lifts via the forgetful functor $\mathrm{For}:Sh_G(Z)\to Sh(Z)$ to $G$-equivariant abelian sheaves $Sh_G(Z)$ on $Z$. 
\end{lem}
\begin{proof}
First of all, the exact forgetful functor $\mathrm{For}:Sh_G(W)\to Sh(W)$ for $W=X,Y,Z$ commutes with the equivariant direct image functor $f^{\mathrm{eq}}_*$, e.g. $\mathrm{For}\circ f^{\mathrm{eq}}_*\simeq f_*\circ \mathrm{For}$. 
Since $Sh_G(W)$ has enough injectives, we can form derived direct image functors, e.g. 
\begin{equation*}
Rf^{\mathrm{eq}}_*:D^b_G(X)\simeq D^b(Sh_G(X))\to D^b_G(Y)\simeq D^b(Sh_G(Y)). 
\end{equation*}
Again they commute with the exact forgetful functors which implies the claim. 
\end{proof}

\begin{prop}\label{p:isoeq}
Let $X=X_b$, $b\in \Bbo$, be a smooth quasi-projective Calabi-Yau threefold with $\Cc$-action as in Section \ref{s:cy}.
Then there is a natural isomorphism 
\begin{equation}
H^3([X/\Cc],\Z)=H^3_{\Cc}(X,\Z)\cong H^3(X,\Z)^{\Cc}
\end{equation}
of $\Z$-MHS. 
\end{prop}
We emphasize that we work over the \emph{integers} giving a stronger result as Lemma \ref{lem:EqCohInv}. 
In general, such a result is false due to torsion, cf. Example \ref{ex:eqcoh} below. 
\begin{proof}
As in the proof of Lemma \ref{lem:EqCohInv}, we employ the spectral sequence (\ref{EqCohSS}). 
In that proof we have seen that its $E_2$-page is of the form
\begin{equation*}
E_2^{pq}=H^p(\Cc, H^q(X,\Z)). 
\end{equation*}
We show that it degenerates on the $E_3$-page in this situation by proving  
\begin{equation}\label{E3deg}
H^p(\Cc,H^q(X,\Z))=0,\quad \forall p\geq 1, q\notin \{3,4  \}. 
\end{equation}
Of course this is automatic if we worked over $\Q$. 
To show it over the integers, we observe that
\begin{gather*}
H^0(X,\Z)\cong H^0(\cu, \pi_*\Z),\quad H^1(X,\Z)\cong H^1(\cu,\pi_*\Z), \quad H^2(X,\Z)\cong H^2(\cu,\pi_*\Z), \\
H^3(X,\Z)\cong H^1(\cu, R^2\pi_*\Z), \quad H^4(X,\Z)\cong H^2(\cu, R^2\pi_*\Z), \quad H^q(X,\Z)=0 \quad\forall q\geq 5
\end{gather*}
for the projection $\pi: X\to \cu$. 
This is seen by using the Leray spectral sequence (cf. \cite{DDP}, \cite{Beck}). 
Since $\Cc$ acts trivially on $\pi_*\Z$, it follows that $\Cc$ acts trivially on $H^q(X,\Z)$ for $q\notin \{ 3,4\}$ by Lemma \ref{lem:leraygroup}. 
This yields (\ref{E3deg}) so that the spectral sequence (\ref{EqCohSS}) gives the isomorphism
\begin{equation*}
E_3^{0,3}=H^3(X,\Z)^\Cc \cong E_\infty^3=H^3_{\Cc}(X,\Z)
\end{equation*}
of $\Z$-MHS. 
\end{proof}
\begin{ex}\label{ex:eqcoh}
The analogous statement is \emph{false} for the minimal resolution $\hat{Y}\to Y$ of the $\Delta$-singularity $Y$ (assuming $\Delta\neq \Delta_h$):
A group cohomology computation shows 
\begin{equation*}
H^2_{\Cc}(\hat{Y},\Z)\cong H^2(\hat{Y},\Z)^{\Cc}\oplus \Z/2\Z. 
\end{equation*}
\end{ex}
Since $H^3(X_b,\Z)$, $b\in \Bbo$, is up to a Tate twist a polarizable $\Z$-HS (see \cite{Beck}, Lemma 5) of weight $1$, it follows that the orbifold intermediate Jacobian
\begin{equation}\label{eq:eqj2}
J^2([X_b/\Cc])=H^3([X_b/\Cc],\C)/\left(F^2H^3([X_b/\Cc],\C)+H^3([X_b/\Cc],\Z)\right)
\end{equation}
of the Calabi-Yau orbifold $[X_b/\Cc]$ is an abelian variety. 
\\
To globalize the previous discussion, we consider the augmented simplicial morphism associated to the family $\pib_{\Cc}:[\X/\Cc]\to \Bb$ of Calabi-Yau orbifolds. 
By abuse of notation, it is again denoted by
\begin{equation*}
\begin{tikzcd}
\pib_{\Cc}:[\X/\Cc]_\bullet \ar[r, "\pib_{\Cc,\bullet}"] & \Bb_\bullet \ar[r, "a"] & \Bb.
\end{tikzcd}
\end{equation*}
where the last arrow is the augmentation. 
In accordance with (\ref{eq:eqcohorb}), we set
\begin{equation*}
R^k\pib_{\Cc,*} \Z= \bm{s} R^k(\pib_{\Cc,\bullet})_*\Z,
\end{equation*}
compare \cite{DeligneIII}, Section 5.2, where $\bm{s}$ stands for the total complex. 

\begin{prop} 
Let $\pib_{\Cc}^\circ:[\X^\circ/\Cc] \to \Bbo$ be the family of orbifolds associated with the smooth family $\pib:\X^\circ\to \Bbo$. 
Then $R^3\pib_{\Cc,*}^\circ \Z$ carries the structure of a polarizable $\Z$-VMHS such that 
\begin{equation*}
(R^3\pib_{\Cc,*}^\circ \Z)_b\cong H^3_{\Cc}(X_b,\Z), \quad b\in \Bbo,
\end{equation*}
naturally as $Z$-MHS. 
\end{prop}
\begin{proof} 
This is a consequence of the more general treatment in \cite{BeckVMHS} on VMHS for simplicial smooth and quasi-projective morphisms (which are topologically locally trivial on each level). 
Alternatively, the V(M)HS-structure can be constructed directly in this present case by using an approximation of the classifying space $B\Cc$ of $\Cc$ by finite-dimensional smooth projective varieties. 
\end{proof}
Therefore the orbifold intermediate Jacobians (\ref{eq:eqj2}) fit into the family 
\begin{equation*}
\mathcal{J}^2([\X^\circ/\Cc])\to \Bbo
\end{equation*}
of abelian varieties. 
Using the methods of \cite{Beck}, it is seen to be an algebraic integrable system. 
\begin{thm}
Let $\Delta$ be an irreducible Dynkin diagram, $\Cc=AS(\Delta)$ and $G=G(\Delta)$ the corresponding simple adjoint complex Lie group. 
Further let $\X\to \Bb(\cu, \Delta)$ be one of the families of Calabi-Yau threefolds of Section \ref{s:cy} with $\Cc$-action. 
Then 
\begin{equation*}
\mathcal{J}^2([\Xo/\Cc])  \cong \Hig_1^\circ(\cu, G)
\end{equation*}
\begin{equation*}
\mathcal{J}_2([\Xo/\Cc]) \cong \Hig_1^\circ(\cu, ^L G).
\end{equation*}
as algebraic integrable systems over $\Bbo(\cu, \Delta)$. 
\end{thm}
\begin{proof}
The spectral sequence (\ref{EqCohSS}) globalizes to a spectral sequence
of polarizable $\Z$-VMHS. 
Hence Proposition \ref{p:isoeq} implies  
\begin{equation*}
R^3\pib_{\Cc,*}^\circ \Z \cong (R^3\pib_*\Z)^\Cc
\end{equation*}
as polarizable $\Z$-VHS of weight $1$ (up to a Tate twist). 
In particular, 
\begin{equation*}
\mathcal{J}^2([\Xo/\Cc]) \cong \mathcal{J}^2_{\Cc}(\Xo)
\end{equation*}
over $\Bbo$ where the right-hand side is the intermediate Jacobian fibration defined by $\Cc$-invariants in cohomology. 
Now the claim follows from \cite{Beck}, Theorem 6. 
\\
Replacing $\pib_{\Cc,*}$ with $\pib_{\Cc,!}$ and using Theorem 8 of \cite{Beck} gives the second isomorphism. 
\end{proof}

%

\appendix

\section{Folding}\label{SectFolding}
Let $\Delta$ be an irreducible Dynkin diagram. 
We follow \cite{Slo} and define the associated symmetry group of $\Delta$ via
\begin{equation}\label{ASDelta}
AS(\Delta):=
\begin{cases}
1, & \Delta\mbox{ is of type }\ADE, \\
\Z/2\Z, & \Delta\mbox{ is of type }\mathrm{B}_{k}, \mathrm{C}_{k}, \mathrm{F}_4, \\
S_3, & \Delta\mbox{ is of type }\mathrm{G}_2,
\end{cases}
\end{equation}
for $k\geq 2$. 
There is a unique irreducible $\ADE$-Dynkin diagram $\Delta_h$ such that $AS=AS(\Delta)\subset \Aut(\Delta_h)$ and $\Delta=\Delta_h^{AS}$.
Here $\Delta_h^{AS}$ stands for the Dynkin diagram which is obtained by taking $AS(\Delta)$-invariants $R_h^{AS}$ in the root space $R_h=R(\Delta_h)$ associated with $\Delta_h$ (cf. \cite{Beck-thesis}, Chapter 1.2, or \cite{Springer})
\begin{rem}\label{r:dynkin}
To obtain a reasonable notion of folding on the level of root spaces, it is important that we only work with Dynkin graph automorphisms $\Aut_D(\Delta_h)\subset \Aut(\Delta_h)$. 
These are graph automorphisms $a\in \Aut(\Delta_h)$ such that $a(v)$ and $v$ are not direct neighbors for each vertex $v\in \Delta_h$. 
If $\Delta_h\neq \mathrm{A}_{2n}$, then $\Aut_D(\Delta_h)=\Aut(\Delta_h)$ and $\Aut_D(\Delta_h)=1$ if $\Delta_h=\mathrm{A}_{2n}$.  
\end{rem}
Restricted to Dynkin diagrams of type $\mathrm{B}_k, \mathrm{C}_k, \mathrm{F}_4, \mathrm{G}_2$ (\emph{$\BCFG$-Dynkin diagrams for short}), we obtain a bijection
\begin{equation}\label{FoldingBijection}
\begin{aligned}
\{ \Delta ~\text{ of type }\BCFG\} &\to \{ (\Delta_h,\Cc)~|~\Delta_h~\ADE,~1\neq \Cc\subset \Aut_D(\Delta_h)\}\\
\Delta&\mapsto (\Delta_h,AS(\Delta)) \\
\Delta=\Delta^{\Cc}_h&\mapsfrom (\Delta_h, \Cc).
\end{aligned}
\end{equation}
We say that $\Delta=\Delta_{h}^\Cc$ is obtained from $(\Delta_h,\Cc)$ (or simply $\Delta_h$) by \emph{folding}. 
For convenience we summarize the corresponding types in the following table 
\begin{equation}\label{FoldingDynkin}
\begin{array}{c|c|c} 
\Delta & \Delta_h & AS(\Delta)  \\ \hline
\mathrm{B_{k+1}} & \mathrm{A}_{2k+1} & \Z/2\Z   \\ 
\mathrm{C}_{k} & \mathrm{D}_{k+1} & \Z/2\Z  \\ 
\mathrm{F}_4 & \mathrm{E}_6 & \Z/2\Z  \\ 
\mathrm{G}_2 & \mathrm{D}_4 & S_3  \\ 
\end{array}
\end{equation}
Finally, we collect the weights of the $\ADE$-singularities that are naturally induced by their Lie-theoretic realization. 
\begin{center}
\begin{tabular}{r l  | c} 
 & Dynkin type of $\Gamma$ and equation & $(w_{r+1},w_{r+2},w_{r+3}; d)$ \\ \hline
$\mathrm{A}_k$: & $x^{k+1}-yz=0$ & $(2,k+1,k+1; 2(k+1))$    \\
$\mathrm{D}_k$: & $x(x^{k-2}-y^2)-z^2=0$ & $(2,k-2,k-1;2k-2)$ \\
$\mathrm{E}_6$: & $x^4+y^3+z^2=0$ & $(6,8,12;24)$   \\
$\mathrm{E}_7$: & $x^3y+y^3+z^2=0$ & $(8,12,18;36)$   \\
$\mathrm{E}_8$: & $x^5+y^3+z^2=0$ & $(12,20,30;60)$
\end{tabular}
\captionof{table}{\footnotesize{$\ADE$-singularities together with their natural weights}}
\label{ADEsings}
\end{center}
Here $r$ is the rank of the corresponding simple complex Lie algebra. 
Note that by definition, this table contains the natural weights of all $\Delta$-singularities. 

\section{Slodowy slices for $\mathrm{B}_2$-singularities}\label{a:ex}
Let $\Delta=\mathrm{B}_2$ and\footnote{Here we use the notation $\Delta_0$ instead of $\Delta_h$ in the main text for notational reasons.} $(\Delta_0,\Cc)=(\mathrm{A}_3,\Z/2\Z)$ be the associated pair. 
Further we let 
 	\begin{equation*}
	\gfr_0=\gfr_(\Delta_0)=\mathfrak{sl}(4,\C), \quad \gfr=\gfr(\Delta)=\mathfrak{so}(5,\C)
	\end{equation*}
be the corresponding simple complex Lie algebras. 
We give two Slodowy slices $S_0\subset \gfr_0$ and $S\subset \gfr$ together with the $\CcA$- and $\Cc$-action respectively (see (\ref{eq:Cc} and (\ref{eq:CA}) and directly compute that they both realize a semi-universal deformation of the $\mathrm{B}_2$-singularity. 

	\subsection{Extrinsic case ($\gfr_0$)}
	We consider the following $\mathfrak{sl}_2$-triple in $\gfr_0$:
	\begin{equation}
	x_0=
	\begin{pmatrix}
	0 & 0 & 0 & 0 \\
	0 & 0 & 1 & 0 \\
	0 & 0 & 0 & 1 \\
	0 & 0 & 0 & 0 \\
	\end{pmatrix},\quad  
	y_0=
	\begin{pmatrix}
	0 & 0 & 0 & 0 \\
	0 & 0 & 0 & 0 \\
	0 & 1 & 0 & 0 \\
	0 & 0 & 1 & 0 \\
	\end{pmatrix},\quad 
	h_0=
	\begin{pmatrix}
	0 & 0 & 0 & 0 \\
	0 & 1 & 0 & 0 \\
	0 & 0 & 0 & 0 \\
	0 & 0 & 0 & -1 \\
	\end{pmatrix}.
	\end{equation}
	The Slodowy slice with respect to $(x_0,y_0,h_0)$ is given by 
	\begin{equation}\label{A3S0}
	S_0=x_0+\ker \mathrm{ad}(y_0)=\left\{
	\begin{pmatrix}
	-3a & b & 0 & 0 \\
	0 & a & 1 & 0 \\
	0 & c & a & 1 \\
	d & e & c & a 
	\end{pmatrix}~:~a,b,c,d,e\in \C \right\}
	\end{equation}
	Let $\phi\in \mathrm{Aut}(\gfr_0)$ be given by $\phi(A)=-A^t$ 
	and
	\begin{equation}
	g_0=
	\begin{pmatrix}
	1 & 0 & 0 & 0 \\
	0 & 0 & 0 & -1 \\
	0 & 0 & 1 & 0 \\
	0 & -1 & 0 & 0 
	\end{pmatrix}
	\end{equation}
	(actually the class in $G_0=PGL(4,\C)$). 
	Then it is immediate to check that 
	\[
	\tau:=\mathrm{Ad}(g_0)\circ \phi\in CA(x,h)
	\] 
	and $\tau^2=1$.
	It generates the subgroup $\mathbf{CA}\subset CA(x,h)$ which is isomorphic to $\Z/2\Z$. 
	Its action on $S_0$ is given by
	\begin{equation}\label{CAaction}
	\tau\cdot \begin{pmatrix}
	-3a & b & 0 & 0 \\
	0 & a & 1 & 0 \\
	0 & c & a & 1 \\
	d & e & c & a 
	\end{pmatrix}
	=
	\begin{pmatrix}
	3a & d & 0 & 0 \\
	0 & -a & 1 & 0 \\
	0 & c & -a & 1 \\
	b & -e & c & -a
	\end{pmatrix}.
	\end{equation}
	To determine its action on the base, we compute $\sigma_0:S_0\to \tfr_0/W_0$. 
	We take as homogeneous generators the coefficients $\chi_2,\chi_3,\chi_4$ of degree $2,3,4$ of the characteristic polynomial. 
	If $B\in S_0$ is as in (\ref{A3S0}), then 
	\begin{equation}
	\sigma_0(B)=(\chi_2(B),\chi_3(B),\chi_4(B))
	= (-6a^2-2c,8a^3-4ac-e,-3a^4+6a^2c-bd-3ae).
	\end{equation}
	It can be readily checked that $\sigma^{-1}(0,0,0)$ is an $\mathrm{A}_3$-singularity given by 
	\begin{equation}\label{A3inS0}
	c=-3a^2,\quad e=20a^3,\quad -81a^4-bd=0. 
	\end{equation}
	in the notation of (\ref{A3S0}).
	Moreover, we have 
	\begin{equation}
	\sigma_0(\tau\cdot B)=(\chi_2(B),-\chi_3(B),\chi_4(B)). 
	\end{equation}
	Hence if we restrict to the subspace 
	\begin{equation*}
	\begin{tikzcd}
	(\tfr_0/W_0)^{\mathbf{CA}}=\{(s,0,t)\in \C^3\cong \tfr_0/W_0~|~s,t\in \C\}\ar[r, hook, "i"] & \tfr_0/W_0,
	\end{tikzcd}
	\end{equation*}	
	 then $\mathbf{CA}$ preserves the fibers of $\sigma_0$. 
	Since $i^*S_0\subset S_0$ is defined by $e=4ac-8a^3$, we obtain
	\begin{equation}
	\sigma_0(B)=(-6a^2-2c,0,-27a^4+18a^2c-bd),\quad B\in i^*S_0.
	\end{equation}
	Finally, we see that $\mathbf{CA}$ acts on the $\mathrm{A}_3$-singularity $\sigma_0^{-1}(0,0,0)$ as described in Example \ref{ex:B2} by setting $x=-3a,y=b,z=d$ in the notation of (\ref{A3inS0}). 
	
	\subsection{Intrinsic case ($\gfr$)}
	For this example, we use the conventions and methods, e.g. the classification of nilpotent orbits in terms of weighted Dynkin diagrams, of \cite{CollingwoodMcGovern}, Chapter 5. 
	First of all, let
	\begin{equation}
	\gfr=\mathfrak{so}(5,\C)\cong
	\left\{
	\begin{pmatrix}
	0 & u_1 & u_2 & v_1 & v_2 \\
	-v_1 & a_1 & a_2 & 0 & b \\
	-v_2 & a_3 & a_4 & -b & 0 \\
	-u_1 & 0 & c & -a_1 & -a_3 \\
	-u_2 & -c & 0 & -a_2 & -a_4
	\end{pmatrix}~:~ 
	u_i,v_j,a_k,b,c\in \C
	\right\}
	\end{equation}
	This form of $\mathfrak{so}(5,\C)$ has the advantage that diagonal matrices give a Cartan $\tfr\subset \gfr$.
	Then we consider the following $\mathfrak{sl}_2$-triplet $(x,y,h)$ where $x\in \gfr$ is a subregular nilpotent element:
	\begin{equation}
	x=\begin{pmatrix}
	0 & 0 & 0 & 1 & 0 \\
	-1 & 0 & 1 & 0 & 0 \\
	0 & 0 & 0 & 0 & 0 \\
	0 & 0 & 0 & 0 & 0 \\
	0 & 0 & 0 & -1 & 0
	\end{pmatrix},\quad
	y=\begin{pmatrix}
	0 & -2 & 0 & 0 & 0 \\
	0 & 0 & 0 & 0 & 0 \\
	0 & 0 & 0 & 0 & 0 \\
	2 & 0 & -2 & 0 & 0 \\
	0 & 2 & 0 & 0 & 0
	\end{pmatrix},\quad
	h=\begin{pmatrix}
	0 & 0 & 0 & 0 & 0 \\ 
	0 & 2 & 0 & 0 & 0 \\
	0 & 0 & 0 & 0 & 0 \\
	0 & 0 & 0 & -2 & 0 \\
	0 & 0 & 0 & 0 & 0 \\
	\end{pmatrix}.
	\end{equation}
	The corresponding Slodowy slice $S=x+\ker \mathrm{ad}~y\subset \gfr$ is given by 
	\begin{equation}\label{SB2}
	S=
	\left\{C=
	\begin{pmatrix}
	0 & - a & -b & 1 & 0 \\
	-1 & 0 & 1 & 0 & 0 \\
	0 & c & -b & 0 & 0 \\
	a & 0 & d & 0 & -c \\
	b & -d & 0 & -1 & b
	\end{pmatrix}~:~a,b,c,d\in \C
	\right\}
	\end{equation}
	Again using the coefficients $\chi_2,\chi_4$ of the characteristic polynomial of degree $2$ and $4$ respectively, the adjoint quotient restricted to $S$ is given by 
	\begin{equation}
	\sigma(C)=(\chi_2(C),\chi_4(C))=(-b^2-2a-2c,2ab^2+2b^2c+2ac+c^2+2cd).
	\end{equation}
	In the notation of (\ref{SB2}), we see that $\sigma^{-1}(0,0,0)$ is given by 
	\begin{equation}\label{B2eq}
	a=-c-\frac{b^2}{2},\quad b^4+c(b^2-2d+c)=0
	\end{equation}
	which is immediately seen to be an $\mathrm{A}_3$-singularity. 
	\\  
	Using the fact that $Z_G(h)$ is the subgroup in $G\cong PSO(5,\C)$ generated by the diagonal maximal torus $T$ and the unipotent subgroup $U_{\alpha_2}$ corresponding to the short root $\alpha_2\in \Delta$, we compute
	\begin{equation}
	C(x,h)=Z_G(x)\cap Z_G(h)=
	\left\{
	D^k=
	\begin{pmatrix}
	1 & 0 & 0 & 0 & 2 \\
	0 & -1 & 0 & 0 & 0 \\
	2 & 0 & -1 & 0 & 2 \\
	0 & 0 & 0 & -1 & 0 \\
	0 & 0 & 0 & 0 & -1
	\end{pmatrix}^k
	~:~k=0,1
	\right\}
	\end{equation}
	(actually the class in $G$). 
	Observe that in this case $\mathbf{C}= C(x,h)$, cf. (\ref{eq:Cc}). 
	Finally, the action of $\Cc$ on $S$ is given by 
	\begin{equation}
	D\cdot C=
	\begin{pmatrix}
	0 & a+2d & b & 1 & 0 \\
	-1 & 0 & 1 & 0 & 0 \\
	0 & 2a+c+2d & b & 0 & 0 \\
	-a-2d & 0 & d & 0 & -2a-c-2d \\
	-b & -d & 0 & -1 & -b 
	\end{pmatrix}.
	\end{equation}
	It is readily seen from (\ref{B2eq}) that $(\sigma^{-1}(0),\mathbf{C})$ is indeed a $\mathrm{B}_2$-singularity. 

\bibliographystyle{alpha}
\bibliography{bibtex1}

\end{document}